\documentclass[12pt,draft]{article}
\usepackage[dvipdfm, letterpaper, margin=2cm]{geometry}
\usepackage{amsmath}%
\usepackage{amsthm}
\usepackage{amssymb}%
\usepackage{amsfonts}
\usepackage{color}
\usepackage{mathrsfs}
\usepackage{framed}
\definecolor{shadecolor}{rgb}{0.95,0.95,0.95}
\allowdisplaybreaks[4]
\numberwithin{equation}{section}
\newtheorem{theorem}{Theorem}[section]
\newtheorem{proposition}[theorem]{Proposition}
\newtheorem{corollary}[theorem]{Corollary}
\newtheorem{lemma}[theorem]{Lemma}
\theoremstyle{definition}
\newtheorem{remark}[theorem]{Remark}
\newtheorem{example}[theorem]{Example}
\newtheorem{definition}[theorem]{Definition}

\DeclareMathOperator{\End}{End}

\DeclareMathOperator{\Hom}{Hom}

\DeclareMathOperator{\Res}{Res}
\DeclareMathOperator{\Sl}{\mathfrak{sl}}

\newcommand{\BF}{\mathbb{F}}
\newcommand{\BN}{\mathbb{N}}
\newcommand{\BZ}{\mathbb{Z}}

\newcommand{\BC}{\mathbb{C}}
\newcommand{\spanf}{\mathrm{span}_\BF}
\newcommand{\Vir}{\mathrm{Vir}}
\newcommand{\fg}{\mathfrak{g}}

\newcommand{\D}{\mathcal{D}}
\newcommand{\B}{\mathcal{B}}

\newcommand{\bc}{\mathbf{c}}
\newcommand{\bk}{\mathbf{k}}

\newcommand{\1}{\mathbf{1}}
\newcommand{\CH}{\mathcal{H}}
\makeatletter
\newtoks\@enLab  
\def\@enQmark{?}
\def\@enLabel#1#2{%
  \edef\@enThe{\noexpand#1{\@enumctr}}%
  \@enLab\expandafter{\the\@enLab\csname the\@enumctr\endcsname}%
  \@enloop}
\def\@enSpace{\afterassignment\@enSp@ce\let\@tempa= }
\def\@enSp@ce{\@enLab\expandafter{\the\@enLab\space}\@enloop}
\def\@enGroup#1{\@enLab\expandafter{\the\@enLab{#1}}\@enloop}
\def\@enOther#1{\@enLab\expandafter{\the\@enLab#1}\@enloop}
\def\@enloop{\futurelet\@entemp\@enloop@}
\def\@enloop@{%
  \ifx A\@entemp         \def\@tempa{\@enLabel\Alph  }\else
  \ifx a\@entemp         \def\@tempa{\@enLabel\alph  }\else
  \ifx i\@entemp         \def\@tempa{\@enLabel\roman }\else
  \ifx I\@entemp         \def\@tempa{\@enLabel\Roman }\else
  \ifx 1\@entemp         \def\@tempa{\@enLabel\arabic}\else
  \ifx \@sptoken\@entemp \let\@tempa\@enSpace         \else
  \ifx \bgroup\@entemp   \let\@tempa\@enGroup         \else
  \ifx \@enum@\@entemp   \let\@tempa\@gobble          \else
                         \let\@tempa\@enOther
             \fi\fi\fi\fi\fi\fi\fi\fi
  \@tempa}
\newlength{\@sep} \newlength{\@@sep}
\providecommand{\sfbc}{\rmfamily\upshape}
\providecommand{\sfn}{\rmfamily\upshape}
\def\@enfont{\ifnum \@enumdepth >1\let\@nxt\sfn \else\let\@nxt\sfbc \fi\@nxt}
\def\enumerate{%
   \ifnum \@enumdepth >3 \@toodeep\else
      \advance\@enumdepth \@ne
      \edef\@enumctr{enum\romannumeral\the\@enumdepth}\fi
   \@ifnextchar[{\@@enum@}{\@enum@}}
\def\@@enum@[#1]{%
  \@enLab{}\let\@enThe\@enQmark
  \@enloop#1\@enum@
  \ifx\@enThe\@enQmark\@warning{The counter will not be printed.%
   ^^J\space\@spaces\@spaces\@spaces The label is: \the\@enLab}\fi
  \expandafter\edef\csname label\@enumctr\endcsname{\the\@enLab}%
  \expandafter\let\csname the\@enumctr\endcsname\@enThe
  \csname c@\@enumctr\endcsname7
  \expandafter\settowidth
            \csname leftmargin\romannumeral\@enumdepth\endcsname
            {\the\@enLab\hskip\labelsep}%
  \@enum@}
\def\@enum@{\list{{\@enfont\csname label\@enumctr\endcsname}}%
           {\usecounter{\@enumctr}\def\makelabel##1{\hss\llap{##1}}%
     \ifnum \@enumdepth>1\setlength{\topsep}{\@@sep}\else
           \setlength{\topsep}{\@sep}\fi
     \ifnum \@enumdepth>1\setlength{\itemsep}{0pt plus1pt minus1pt}%
      \else \setlength{\itemsep}{\@@sep}\fi
     \setlength{\parsep}{0pt plus1pt minus1pt}%
     \setlength{\parskip}{0pt plus1pt minus1pt}
                   }}

\def\endenumerate{\par\ifnum \@enumdepth >1\addvspace{\@@sep}\else
           \addvspace{\@sep}\fi \endlist}
\makeatother                                                            %

\begin{document}
\title{Symmetric Invariant Bilinear Forms on Modular Vertex Algebras}
\author{Haisheng Li\\
Department of Mathematical Sciences, Rutgers University,\\
Camden, NJ 08102, USA, and\\
School of  Mathematical Sciences, Xiamen University, Xiamen 361005, China\\
\textit{E-mail address:} \texttt{hli@camden.rutgers.edu}\and
Qiang Mu\\
School of Mathematical Sciences, Harbin Normal University,\\
Harbin, Heilongjiang 150025, China\\
\textit{E-mail address:} \texttt{qmu520@gmail.com}}
\maketitle

\begin{abstract}
In this paper, we study contragredient duals and invariant bilinear forms for modular vertex algebras (in characteristic $p$).
We first introduce a bialgebra $\CH$ and we then
introduce a notion of $\CH$-module vertex algebra and a notion of $(V,\CH)$-module
for an $\CH$-module vertex algebra $V$. Then we give a modular version of Frenkel-Huang-Lepowsky's theory and
study invariant bilinear forms on an $\CH$-module vertex algebra.
As the main results, we obtain an explicit description of the space of invariant bilinear forms on a general $\CH$-module
vertex algebra, and we apply our results to affine vertex algebras and Virasoro vertex algebras.
\end{abstract}

\section{Introduction }
In the theory of vertex operator algebras in characteristic zero, an important notion
is that of contragredient dual (module), which was due to Frenkel, Huang, and Lepowsky (see \cite{FHL}).
Closely related to contragredient dual is the notion of invariant bilinear form on a vertex operator algebra
and it was proved therein that every invariant bilinear form on a vertex operator algebra is automatically symmetric.
Contragredient dual and symmetric invariant bilinear forms have played important roles in various studies.
Note that as part of its structure, a vertex operator algebra $V$ is a natural module
 for the Virasoro algebra. In \cite{FHL},  a notion of quasi vertex operator algebra was introduced, which generalizes that of
vertex operator algebra with an internal Virasoro algebra action replaced by an external $\Sl_{2}$-action, and  the corresponding results were extended for more general quasi vertex operator algebras.
Note that a general vertex algebra $V$ (without conform vector) is naturally a
module for the (Hopf) algebra $\BC[\D]$ where $\D$ is the canonical derivation defined by $\D(v)=v_{-2}{\bf 1}$ for $v\in V$, whereas
we have $L(-1)=\D$ for a quasi vertex operator algebra $V$.

Invariant bilinear forms on vertex operator algebras were furthermore studied in \cite{Li94}
and an explicit description of the space of invariant bilinear forms was obtained.
More specifically,  it was proved that the space of invariant bilinear forms on a vertex operator algebra $V$
 is canonically isomorphic to the space
$\text{Hom}_{V}(V,V')$, where $V'$ is the contragredient dual of the adjoint module $V$.
Furthermore, it was proved that $\text{Hom}_{V}(V,V')$ is isomorphic to $(V_{(0)}/L(1)V_{(1)})^{*}$,
where $V=\bigoplus_{n\in \BZ}V_{(n)}$ (graded by the conformal weight). Consequently, this gives rise to
 a necessary and sufficient condition for the existence of a non-degenerate symmetric invariant bilinear form on a simple vertex operator algebra.
It is conceivable that contragredient duals and bilinear forms will be practically important in the theory of modular vertex algebras also.

In this paper, we study contragredient duals and invariant bilinear forms for modular vertex algebras.
We first introduce a bialgebra $\CH$ by using the universal enveloping algebra of the complex Lie algebra $\Sl_2$ and
introduce a notion of $\CH$-module vertex algebra and a notion of $(V,\CH)$-module
for an $\CH$-module vertex algebra $V$, and we then
formulate FHL's theory in characteristic $p$ and define invariant bilinear form on an $\CH$-module vertex algebra.
As the main results, we obtain an explicit description of the space of invariant bilinear forms on a general $\CH$-module
vertex algebra, and we apply our results to affine vertex algebras and Virasoro vertex algebras.

We now give a more detailed description of the content of this paper.
Let $\BF$ be an algebraically closed field of a prime characteristic $p$.  We assume $p\ne 2$ throughout this paper.
The notion of vertex algebra over $\BF$ was defined  by Borcherds (see \cite{B86}). On the other hand,
it can also be defined equivalently by using
what was called (Cauchy-)Jacobi identity in \cite{FLM} and \cite{FHL}, as it is done in this paper.
For vertex algebras over $\BF$, the bialgebra $\BC[\D]$ in the case of characteristic zero is replaced with a bialgebra $\B$
over $\BF$ (see \cite{B86}), where $\B$ has  a basis $\D^{(n)}$ for $n\in \BN$  with  $\D^{(0)}=1$ and
\begin{equation*}
\D^{(m)}\cdot \D^{(n)}=\binom{m+n}{n}\D^{(m+n)},\   \  \
\Delta(\D^{(n)})=\sum_{j=0}^{n}\D^{(n-j)}\otimes \D^{(j)},\    \    \   \
\varepsilon(\D^{(n)})=\delta_{n,0}
\end{equation*}
for $m,n\in \BN$.  The fact is that every vertex algebra $V$ over $\BF$ is naturally a $\B$-module with
$\D^{(n)}v=v_{-n-1}{\bf 1}$ for $n\in \BN,\ v\in V$.

To study contragredient dual, we introduce a modular counterpart of the notion of quasi vertex operator algebra.
Note that in the case of characteristic zero, (formal) exponentials $e^{zL(\pm 1)}$ and $z^{L(0)}$ play a crucial role
(see \cite{FHL}). For the modular case, we shall need a bialgebra over $\BF$.
Consider the complex Lie algebra $\Sl_{2}$ with a basis
$\{ L_{1},L_0,L_{-1}\}$ such that
\begin{equation*}
[L_{0},L_{\pm 1}]=\mp L_{\pm 1}, \  \   \  \  [L_{1},L_{-1}]=2L_{0}.
\end{equation*}
For any nonnegative integer $n$, set
\begin{equation*}
L_{\pm 1}^{(n)}=\frac{1}{n!}(L_{\pm 1})^{n},\   \   \   \
L_{0}^{(n)}=\binom{-2L_{0}}{n}=\frac{1}{n!}(-2L_0)(-2L_0-1)\cdots (-2L_0-n+1),
\end{equation*}
which are elements of the universal enveloping algebra $U(\Sl_2)$.
Let $U(\Sl_2)_{\BZ}$ denote the subring of  $U(\Sl_2)$ generated by these elements.
From a result of Kostant,  $U(\Sl_2)_{\BZ}$ is an integral form of  $U(\Sl_2)$. Then
define $\CH=\BF\otimes_{\BZ}U(\Sl_2)_{\BZ}$, which is a Hopf algebra over $\BF$.

The modular analogue of a quasi vertex operator algebra is what we call an $\CH$-module vertex algebra.
By definition, an $\CH$-module vertex algebra is a vertex algebra $V$
which is also a lower truncated $\BZ$-graded $\CH$-module such that
$L_{-1}^{(n)}=\D^{(n)}$ on $V$  and
\begin{equation*}
L_{0}^{(n)}v=\binom{-2m}{n}v\   \   \   \mbox{ for }n\in \BN,\  v\in V_{m},\  m\in \BZ,
\end{equation*}
and such that a counterpart of the conjugation relation involving $L_{1}$ of \cite{FHL} holds.
We then define a notion of lower truncated $(V,\CH)$-module
for an $\CH$-module vertex algebra $V$ accordingly.
Under this setting, for any lower truncated $(V,\CH)$-module $W$,
we have the contragredient dual $W'$, which is also a lower truncated $(V,\CH)$-module, where
$W'=\bigoplus_{n\in \BZ}W_{n}^{*}$ as a vector space.
It follows from the same arguments in \cite{FHL} that every invariant bilinear form on an $\CH$-module vertex algebra
$V$ is automatically symmetric. It is proved here that the space of invariant bilinear forms on $V$
is canonically isomorphic to the space $\text{Hom}_{(V,\CH)}(V,V')$. Furthermore, we prove that
for any lower truncated $(V,\CH)$-module $W$,  $\text{Hom}_{(V,\CH)}(V,W)$
 is naturally isomorphic to the subspace $W^{\CH}=\{ w\in W\ |\  L_{\pm 1}^{(n)}w=0\   \   \mbox{ for }n\ge 1\}$.
 As the main result of this paper,
we prove that the space of invariant bilinear forms on $V$ is canonically isomorphic to
the space $(V_{0}/(L_{1}^{+}V)_{0})^{*}$,
where $L_{1}^{+}V=\sum_{n\ge 1}L_{1}^{(n)}V$ and $(L_{1}^{+}V)_{0}$ is the degree-zero subspace.
To achieve this, as a key step we prove that $(L_{-1}^{+}V)_{0}\subset (L_{1}^{+}V)_{0}$.

Note that there are essential differences between the modular case and the characteristic-zero case,
where proofs in modular case are often more complicated.
In the case of characteristic zero, a vertex operator algebra is automatically a quasi vertex operator algebra,
but a vertex operator algebra over $\BF$ (see \cite{DR1}) is not necessarily an $\CH$-module vertex algebra.
On the other hand, we show that affine vertex algebras and Virasoro vertex algebras are naturally $\CH$-module vertex algebras. We also prove that they satisfy the condition that $(L_{1}^{+}V)_{0}=0$.

We note that there have already been several studies in literature on modular vertex operator algebras.
For example, the representation theory, including $A(V)$-theory, was studied in \cite{DR1},
Virasoro vertex operator algebras were studied in \cite{DR2}, and
a modular $A_{n}(V)$ theory was studied in \cite{R}, while framed vertex operator algebras
were studied in \cite{DLR}.
On the other hand, integral forms of some vertex operator algebras were studied
in \cite{BR1,BR2}, \cite{DG}, \cite{GL}, \cite{Mc1,Mc2},
and related modular vertex algebras were used to study modular moonshine in \cite{BR1,BR2}, \cite{GL}.

This paper is organized as follows. Section 2 is preliminaries.
In Section 3, we first introduce the bialgebra $\CH$ and then define the notion of $\CH$-module vertex algebra.
For a $(V,\CH)$-module $W$, we study the contragredient dual $W'$. In Section 4, we study
invariant bilinear forms on $\CH$-module vertex algebras and present the main results.
In Section 5, we consider vertex algebras associated to affine Lie algebras and the Virasoro algebra, and
we show that the space of invariant bilinear forms is $1$-dimensional.

\textbf{Acknowledgement:}
For this research, Li is partially supported by the China NSF (grants 11471268, 11571391, 11671247) and Mu
is partially supported by the China NSF (grant 11571391)
and the Heilongjiang Provincial NSF (grant LC2015001).

\section{Basics}
Let $\BF$ be an algebraically closed field of an odd prime characteristic $p$, which is fixed throughout this paper.
All vector spaces (including algebras) are considered to be over $\BF$.
In addition to the standard usage of $\BZ$ for the integers,
we use  $\BZ_+$ for the positive integers and $\BN$ for the nonnegative integers.

Note that for any $m\in \BZ,\ k\in \BN$,
\begin{equation*}
\binom{m}{k}=\frac{m(m-1)\cdots (m+1-k)}{k!}\in \BZ.
\end{equation*}
Then we shall also view $\binom{m}{k}$ as an element of $\BF$.
Furthermore, for $m\in \BZ$ we have
\begin{equation*}
(x\pm z)^{m}=\sum_{k\ge 0}\binom{m}{k} (\pm 1)^{k}x^{m-k}z^{k}\in \BF[x,x^{-1}][[z]].
\end{equation*}
As in  \cite{LM}, define $\B$ to be the bialgebra with a basis  $\{\D^{(r)}\mid r\in\BN\}$, where
\begin{align*}
    \D^{(m)}\cdot  \D^{(n)}&=\binom{m+n}{n}\D^{(m+n)}, &
    \hspace*{-4em}\D^{(0)}&=1,\\
    \Delta(\D^{(n)})&=\sum_{i=0}^n \D^{(n-i)}\otimes \D^{(i)}, &
    \hspace*{-4em}\varepsilon(\D^{(n)})&=\delta_{n,0}
\end{align*}
for $m,n\in\BN$. We see that $\B$ is an $\BN$-graded algebra with
\begin{equation}
    \deg \D^{(n)}=n\   \   \mbox{ for }n\in \BN.
\end{equation}
Form a generating function
\begin{equation}
    e^{x\D}=\sum_{n\ge 0}x^{n}\D^{(n)}\in \B[[x]].
\end{equation}
The bialgebra structure of $\B$ can be described in terms of the generating functions as
\begin{eqnarray}\label{eq:eDeD-eD}
 e^{x\D}e^{z\D}=e^{(x+z)\D},\qquad
 \Delta(e^{x\D})=e^{x\D}\otimes e^{x\D},\qquad \varepsilon(e^{x\D})=1.
\end{eqnarray}
Especially, we have $ e^{x\D}e^{-x\D}=1$.

\begin{remark}
Let $U$ be any vector space. Then $U[[x,x^{-1}]]$ is a $\B$-module with $\D^{(n)}$ for $n\in \BN$ acting as
the {\em $n$-th Hasse differential operator} $\partial_x^{(n)}$ (with respect to $x$), which is defined by
\begin{equation}
    \partial_x^{(n)}x^m=\binom{m}{n}x^{m-n}\    \   \   \mbox{ for }m\in\BZ.
\end{equation}
Setting  $e^{z\partial_{x}}=\sum_{n\ge 0}z^{n}\partial_x^{(n)}$, we have
\begin{eqnarray}
A(x+z)=e^{z \partial_x}A(x)
\end{eqnarray}
for $A(x)\in U[[x,x^{-1}]]$.
\end{remark}

Recall that a {\em vertex algebra} over $\BF$ (cf. \cite{B86}) is a vector space $V$,
equipped with a vector ${\bf 1}\in V$, called the {\em vacuum vector}, and with a linear map
\begin{align*}
Y(\cdot,x):\ V&\to (\End V)[[x,x^{-1}]],\\
v&\mapsto Y(v,x)=\sum_{n\in \BZ}v_{n}x^{-n-1}\  \ (\mbox{where }v_{n}\in \End V),
\end{align*}
satisfying the following conditions for $u,v\in V$:
\begin{eqnarray*}
&&u_{n}v=0\   \   \   \mbox{ for $n$ sufficiently large},\\
&&Y({\bf 1},x)=1 \   \  (\mbox{where $1$ denotes the identity operator}),\\
&&Y(v,x){\bf 1}\in V[[x]]\  \  \mbox{and }\  \lim_{x\rightarrow 0}Y(v,x){\bf 1}=v,
\end{eqnarray*}
and
\begin{eqnarray}\label{Jacobi-va}
&&x_0^{-1}\delta\left(\frac{x_1-x_2}{x_0}\right)Y(u,x_1)Y(v,x_2)-x_0^{-1}\delta\left(\frac{x_2-x_1}{-x_0}\right)Y(v,x_2)Y(u,x_1)
\nonumber\\
&&\qquad \qquad =x_2^{-1}\delta\left(\frac{x_1-x_0}{x_2}\right)Y(Y(u,x_0)v,x_2)
\end{eqnarray}
(the {\em Jacobi identity}).

We have the following simple fact (see \cite{B86}):

\begin{lemma}
Let $V$ be a  vertex algebra. Then $V$ is a  $\B$-module with
\begin{equation}\label{special-skewsymm}
e^{x\D}v=\sum_{n\ge0}x^{n}\D^{(n)}v=Y(v,x)\1\   \   \   \mbox{ for }v\in V.
\end{equation}
Furthermore,  the following {\em skew-symmetry} holds for $u,v\in V$:
\begin{equation}
Y(u,x)v=e^{x\D}Y(v,-x)u=\sum_{n\ge 0}x^{n}\D^{(n)}Y(v,-x)u.
\end{equation}
\end{lemma}

For a vertex algebra $V$, a {\em $V$-module} is a vector space $W$ equipped with a linear map
\begin{align*}
Y_{W}(\cdot,x):\   V&\rightarrow (\End W)[[x,x^{-1}]],\\
v&\mapsto Y_{W}(v,x)=\sum_{n\in \BZ}v_{n}x^{-n-1},
\end{align*}
satisfying the conditions that $Y_{W}(\1,x)=1$,  $Y_{W}(v,x)w\in W((x))$ for $v\in V,\ w\in W$, and that
 for $u,v\in V$, the corresponding Jacobi identity (\ref{Jacobi-va}) holds.
We also often denote a $V$-module by a pair $(W,Y_W)$.

We have the following analogue of a result in characteristic zero (cf. \cite{LL}):

\begin{lemma}
Let $V$ be a vertex algebra and let  $(W,Y_{W})$ be any $V$-module. Then
\begin{eqnarray}
Y_{W}(e^{z\D}v,x)=Y_{W}(v,x+z)=e^{z\partial_{x}}Y_{W}(v,x)\   \   \  \mbox{ for }v\in V.
\end{eqnarray}
\end{lemma}

Following  \cite{LM}, we introduce a notion of $(V,\B)$-module.

\begin{definition}
Let $V$ be a vertex algebra. A {\em $(V,\B)$-module} is a $V$-module $(W,Y_W)$ which is also
a $\B$-module  such that
\begin{equation*}
    e^{x\D}Y_W(v,z)e^{-x\D}=Y_W(e^{x\D}v,z)\quad \mbox{ for }v\in V.
\end{equation*}
\end{definition}

For  any $(V,\B)$-module $(W,Y_W)$, we have
\begin{equation}\label{V,B-module-property}
    e^{x\D}Y_W(v,z)e^{-x\D}=Y_W(e^{x\D}v,z)=e^{x\partial_z}Y_W(v,z)=Y_W(v,z+x)\quad \mbox{ for }v\in V.
\end{equation}
Notice that the adjoint module of $V$ is automatically a $(V,\B)$-module.

The following result can be found in \cite{LM} (cf. \cite{Li94}):

\begin{lemma}\label{vacuum-like}
Let $V$ be a vertex algebra and let $W$ be a $(V,\B)$-module.
Suppose that $w$ is a vector in $W$ such that $\D^{(n)}w=0$ for $n\ge 1$.
Then
\begin{equation}\label{eq:WYd-vacuumlike-01}
    Y_W(v,x)w=e^{x\D}v_{-1}w\quad\mbox{ for  }v\in V,
\end{equation}
and the linear map $f_{w}:V\to W$ defined by $f_{w}(v)=v_{-1}w$ for $v\in V$ is a $V$-module homomorphism.
\end{lemma}

Let $V$ be a vertex algebra and $W$ a $(V,\B)$-module.
Set
\begin{equation}
    W^\D=\{ w\in W\ |\ \D^{(n)}w=0\  \mbox{ for }n\ge 1\}.
\end{equation}
Recall that $\B$ is a bialgebra with counit map $\varepsilon$. Then we have
\begin{equation*}
W^{\D}=\{ w\in W\ |\  bw=\varepsilon(b)w\  \  \text{for }b\in \B\}.
\end{equation*}
Denote by $\Hom_{(V,\B)}(V,W)$ the set of all $(V,\B)$-module homomorphisms from $V$ to $W$,
which by definition are $V$-module homomorphisms that are also $\B$-module homomorphisms.

\begin{lemma}\label{th:f(1)-in-W^D}
Let $V$ be a vertex algebra and let $W$ be a $(V,\B)$-module.
Assume that $f:V\to W$ is a $V$-module homomorphism.
Then $f\in \Hom_{(V,\B)}(V,W)$  if and only if $f(\1)\in W^{\D}$. Furthermore,
the map $\phi: \Hom_{(V,\B)}(V,W)\rightarrow W^{\D}$ defined by $\phi(f)=f(\1)$
 for $f\in \Hom_{(V,\B)}(V,W)$ is a linear isomorphism.
\end{lemma}

\begin{proof}  Assume $f\in \Hom_{(V,\B)}(V,W)$. As $e^{x\D}\1=\1$ and $fe^{x\D}=e^{x\D}f$,
we have $f(\1)=f(e^{x\D}\1)=e^{x\D}f(\1)$, which implies $\D^{(r)}f(\1)=0$ for $r\ge 1$.
Thus, $f(\1)\in W^{\D}$.
Conversely, assume $f(\1)\in W^{\D}$.
For $v\in V$, using (\ref{special-skewsymm})  and  Lemma~\ref{vacuum-like}, we get
\begin{align*}
    f (e^{x\D}v)&=f (Y(v,x)\1)=Y_W(v,x)f(\1)=e^{x\D} v_{-1}f(\1)=e^{x\D} f(v_{-1}\1)
    =e^{x\D}f(v).
\end{align*}
This proves $fe^{x\D}=e^{x\D}f$, so $f\in \Hom_{(V,\B)}(V,W)$.

It is clear that $\phi$ is linear. If $f\in \ker \phi$, then
$f(v)=f(v_{-1}\1)=v_{-1}f(\1)=v_{-1}\phi(f)=0$ for every $v\in V$, which implies  $f=0$.
 Thus $\phi$ is injective.
Now, let $w\in W^\D$. By Lemma~\ref{vacuum-like}, we have a $V$-module homomorphism
$f_w$ from $V$ to $W$ with $f_{w}(v)=v_{-1}w$ for $v\in V$, where $\phi(f_{w})=f_{w}(\1)=\1_{-1}w=w$.
By the first part, we have $f_w\in \Hom_{(V,\B)}(V,W)$.
This proves that $\phi$ is surjective.
 Therefore, $\phi$ is a linear isomorphism.
\end{proof}

\begin{definition}
A \emph{$\BZ$-graded vertex algebra} is a vertex algebra $V$ equipped with a $\BZ$-grading
$V=\bigoplus_{n\in\BZ}V_n$ such that
\begin{gather}
    \1\in V_0,\label{eq:deg-vacuum}\\
    u_rV_{n}\subset V_{m+n-r-1} \label{eq:deg-V}
\end{gather}
for $u\in V_m$ with $m\in \BZ$ and for $n, r\in\BZ$.
\end{definition}

Note that the condition (\ref{eq:deg-vacuum}) actually follows from (\ref{eq:deg-V}) and the creation property (\ref{special-skewsymm}). On the other hand,
from (\ref{special-skewsymm}),  (\ref{eq:deg-vacuum}) and (\ref{eq:deg-V}), we have
\begin{equation}\label{Ddegree}
    \D^{(r)}V_{n}\subset V_{n+r}\quad\mbox{ for }r\in \BN,\ n\in \BZ.
\end{equation}
Thus, every $\BZ$-graded vertex algebra is automatically a $\BZ$-graded $\B$-module.

Let $V$ be a $\BZ$-graded vertex algebra.  A {\em $\BZ$-graded $V$-module} is a $V$-module equipped with a
$\BZ$-grading $W=\bigoplus_{n\in\BZ} W_n$ such that
\begin{equation}
    v_{m}W_{n}\subset W_{k+n-m-1}\quad\mbox{ for }v\in V_{k},\ k,m,n\in \BZ.
\end{equation}

Let $W=\bigoplus_{n\in\BZ}W_n$ be a $\BZ$-graded vector space.
For $w\in W_n$ with $n\in \BZ$, we set
\begin{equation}
    \deg w=n\in \BZ
\end{equation}
and we say $w$ is \emph{homogeneous of degree $n$}.
For homogeneous vector $w\in W$, we define
\begin{equation*}
    x^{\deg} (w)=x^{\deg w}w\in W[x,x^{-1}].
\end{equation*}
Then extend this definition linearly to get a linear map
\begin{equation}
    x^{\deg}: W\to W[x,x^{-1}].
\end{equation}

\begin{remark}\label{1-x=deg}
Let $W=\bigoplus_{n\in\BZ}W_n$ be a $\BZ$-graded vector space.
From definition, a vector $w\in W$ is homogeneous of degree $n\in \BZ$  if and only if $x^{\deg}w=x^{n}w$.
On the other hand, it is straightforward to show that a vector $w\in W$ is homogeneous of degree $n$ if and only if
\begin{equation*}
(1-x)^{-2\deg}w=(1-x)^{-2n}w.
\end{equation*}
\end{remark}

Let $V=\bigoplus_{n\in\BZ}V_n$ be a $\BZ$-graded vertex algebra. From \cite{FHL} we have
\begin{equation}\label{eq:conj-deg}
    z^{\deg}Y(v,z_0)z^{-\deg}=Y(z^{\deg}v,zz_0)
\end{equation}
for $v\in V$.  Note that relation (\ref{Ddegree}) amounts  to
\begin{equation}
    z^{\deg}\D^{(r)}z^{-\deg}=z^{r}\D^{(r)}
\end{equation}
on $V$ for $r\in \BN$, which can be written in terms of the generating function as
\begin{equation}\label{eq:conj-deg-e^D}
    z^{\deg}e^{x\D}z^{-\deg}=e^{xz\D}.
\end{equation}

\section{$\CH$-module vertex algebras and contragredient duals}

In characteristic $0$, a theory of contragredient modules for vertex operator algebras was established  in \cite{FHL} and
it was also extended for what were  therein called quasi vertex operator algebras.
The notion of quasi vertex operator algebra generalizes the notion of vertex operator algebra in the way that
a natural module structure on every vertex operator algebra for the Virasoro algebra
is reduced to a suitable module structure for Lie algebra $\Sl_{2}$ on every quasi vertex operator algebra.
 In this section, we shall study Frenkel-Huang-Lepowsky's theory in characteristic $p$.

We begin to introduce a bialgebra $\CH$.
Consider the $3$-dimensional simple Lie algebra $\Sl_2$ over $\BC$
with a basis 
$\{L_{-1},L_0,L_1\}$ such that
\begin{equation*}
    [L_{1},L_{-1}]=2L_0,\    \    \   \   [L_0, L_{\pm 1}]=\mp L_{\pm 1}.
\end{equation*}
The universal enveloping algebra $U(\Sl_2)$ is naturally a Hopf algebra, especially a bialgebra over $\BC$.
View $U(\Sl_2)$ as a $\BZ$-graded algebra with
\begin{eqnarray}
\deg L_{\pm 1}=\mp 1,\   \   \   \   \deg L_{0}=0.
\end{eqnarray}
Note that the following relations hold in  $U(\Sl_2)$ (see \cite{FHL}):
\begin{align}
    z^{L_{0}}e^{xL_{\pm 1}}&=e^{xz^{\mp 1}L_{\pm 1}}z^{L_0},\label{FHL-1}\\
    e^{xL_1}e^{zL_{-1}}&=e^{(1-xz)^{-1}zL_{-1}}(1-xz)^{-2L_0}e^{(1-xz)^{-1}xL_1}.\label{FHL-2}
\end{align}

For $n\in\BN$, set
\begin{equation}
    L_{\pm 1}^{(n)}=\frac{L_{\pm 1}^n}{n!},\   \   \    \
   L_{0}^{(n)}= \binom{-2L_0}{n}=\frac{(-2L_0)(-2L_0-1)\cdots(-2L_0-n+1)}{n!}
\end{equation}
in $U(\Sl_2)$.
Denote by $U(\Sl_2)_\BZ$ the $\BZ$-span  in $U(\Sl_2)$  of elements
\begin{equation}\label{eUbasis}
    L_{-1}^{(i)}L_0^{(j)}L_1^{(k)}
\end{equation}
for $i,j,k\in\BN$. Note that $U(\Sl_2)_\BZ$ is also the subring generated by $L_{\pm 1}^{(n)},\   \   L_{0}^{(n)}$ for
$n\in \BN$.
According to a theorem of Kostant (cf.~\cite{hum}), $U(\Sl_2)_\BZ$ is an integral form of $U(\Sl_2)$
viewed as a Hopf algebra.
The following relations hold in $U(\Sl_2)_{\BZ}$ for $m,n\in\BN$:
\begin{align}
L_0^{(m)}L_0^{(n)}&=\sum_{j=0}^{m}\binom{m}{j} \binom{n+j}{m}L_0^{(n+j)},\label{e2L0-multiplication}\\
L_{\pm 1}^{(m)}L_{\pm 1}^{(n)}&=\binom{m+n}{m}L_{\pm 1}^{(m+n)}, \label{eq-1}\\
      L_0^{(m)} L_{\pm 1}^{(n)}&=L_{\pm 1}^{(n)}\binom{-2L_0\pm 2n}{m}
    =\sum_{i=0}^{m}\binom{\pm 2n}{i}L_{\pm1}^{(n)}L_{0}^{(m-i)}, \label{eq-3}\\
L_{1}^{(m)}L_{-1}^{(n)}&=\sum_{i=0}^{\min(m,n)}\sum_{j=0}^{i}\binom{-m-n+2i}{j}(-1)^iL_{-1}^{(n-i)}L_{0}^{(i-j)}L_{1}^{(m-i)},\label{eq-5}\\
   L_{\pm 1}^{(m)}L_{0}^{(n)}&=\sum_{i=0}^{n}\binom{-2m}{i}L_{0}^{(n-i)}L_{\pm 1}^{(m)}.\nonumber
\end{align}
For the coalgebra structure of $\CH$, we have
\begin{align}
    \varepsilon(L_{\pm 1}^{(n)})&=\delta_{n,0},&\varepsilon (L_{0}^{(n)})&=\delta_{n,0},\label{coalgebra1}\\
    \Delta(L_{\pm 1}^{(n)})&=\sum_{i=0}^{n}L_{\pm 1}^{(n-i)}\otimes L_{\pm 1}^{(i)},&
    \Delta (L_{0}^{(n)})&=\sum_{i=0}^{n}L_0^{(n-i)}\otimes L_0^{(i)}\label{coalgebra2}
\end{align}
for $n\in \BN$, noticing that
\begin{align}
\Delta \binom{-2L_0}{n}=\binom{-2L_0\otimes 1-1\otimes 2L_0}{n}
=\sum_{i=0}^{n}\binom{-2L_0\otimes 1}{n-i}
\binom{-1\otimes 2L_0}{i}.
\end{align}

Now, we define
\begin{equation}
\CH=\BF\otimes_{\BZ} U(\Sl_2)_\BZ,
\end{equation}
which is a Hopf algebra over $\BF$. It follows that
the elements in (\ref{eUbasis}) give rise to a basis of $\CH$ and the bialgebra structure is given by
(\ref{e2L0-multiplication}--\ref{eq-5}), (\ref{coalgebra1}), and (\ref{coalgebra2}).


In fact,  one can show that $\CH$ is isomorphic to the associative algebra with generators
\begin{equation*}
L_{-1}^{(n)},\   \   L_{0}^{(n)},\   \   L_{1}^{(n)} \   \   (n\in \BN),
\end{equation*}
subject to the relations above and the following relation
\begin{align}
    L_{-1}^{(0)}&=L_{1}^{(0)}=L_{0}^{(0)}=1,   \label{eq:H1}
 \end{align}

\begin{remark}\label{anti-automorphism}
Note that Lie algebra $\Sl_2$ admits an order-$2$ automorphism  $\sigma$  given by
\begin{eqnarray}
\sigma(L_{\pm 1})=L_{\mp 1},\    \   \   \    \sigma(L_0)=-L_0.
\end{eqnarray}
This naturally gives rise to an automorphism  of $\CH$, which is also denoted by $\sigma$, such that
\begin{eqnarray}
\sigma(L_{\pm 1}^{(n)})=L_{\mp 1}^{(n)},\   \   \   \   \sigma(L_0^{(n)})=(-1)^{n}\sum_{i=0}^{n}\binom{n-1}{i}L_0^{(n-i)}
\end{eqnarray}
for $n\in \BN$.
On the other hand,  Lie algebra $\Sl_2$ admits an anti-automorphism  $\theta$  given by
\begin{equation*}
\theta(L_{\pm 1})=L_{\mp 1},\    \   \   \    \theta(L_0)=L_0.
\end{equation*}
It then follows that
$\CH$ admits an anti-automorphism also denoted by  $\theta$  such that
\begin{eqnarray}
\theta(L_{\pm 1}^{(n)})=L_{\mp 1}^{(n)},\    \   \     \    \   \   \theta(L_0^{(n)})=L_0^{(n)}
\   \   \   \mbox{ for }n\in \BN.
\end{eqnarray}
\end{remark}

Set
\begin{eqnarray}
\CH_{0}=\sum_{n\in \BN}\BF L_0^{(n)},\    \   \   \    \CH_{\pm}=\sum_{n\in \BN}\BF L_{\pm 1}^{(n)}.
\end{eqnarray}
These are subalgebras of $\CH$ and in fact,
 $\{ L_{0}^{(n)}\ |\ n\in \BN\}$ is a basis of $\CH_{0}$,
$\{L_{\pm 1}^{(r)}\ |\ r\in\BN\}$ is a basis of $\CH_{\pm}$.
Both $\CH_{+}$ and $\CH_{-}$ are isomorphic to the bialgebra $\B$ and
we have $\CH =\CH_{+}\otimes \CH_{0}\otimes \CH_{-}$.

Note that $\CH$ is a $\BZ$-graded algebra with
\begin{equation}
\deg L_{\pm 1}^{(n)}=\mp n,\  \   \     \  \deg L_0^{(n)}=0\    \   \   \mbox{ for }n\in \BN.
\end{equation}
We then define the notion of $\BZ$-graded $\CH$-module in the obvious way. Furthermore, we have:

\begin{definition}\label{def-graded-H-module}
A \emph{$\BZ$-graded weight $\CH$-module} is a $\BZ$-graded $\CH$-module $W=\bigoplus_{n\in\BZ}W_n$
on which  $L_{0}^{(r)}$ acts as $\binom{-2\deg}{r}$ for $r\in\BN$, i.e.,
\begin{eqnarray}
L_0^{(r)}|_{W_{n}}=\binom{-2n}{r}\   \   \mbox{ for }r\in\BN,\ n\in\BZ.
\end{eqnarray}
\end{definition}

\begin{remark}\label{-2deg=L0}
Let  $W$ be a $\BZ$-graded weight $\CH$-module. Then we have
\begin{align}\label{deg=L0}
(1+z)^{-2\deg}w=\sum_{i\ge 0}\binom{-2\deg}{i}wz^{i}=\sum_{i\ge 0}L_{0}^{(i)}wz^{i}
\end{align}
for $w\in W$.
\end{remark}

\begin{example}\label{Laurent}
Consider $\BF[x,x^{-1}]$ as a $\BZ$-graded algebra with  $\deg x^{m}=-m$ for $m\in \BZ$.
Then $\BF[x,x^{-1}]$ is a $\BZ$-graded weight $\CH$-module algebra with
\begin{eqnarray}
L_{-1}^{(r)} \cdot x^{m}=(-1)^{r}\binom{m}{r}x^{m-r},\    \    \   L_{1}^{(r)} \cdot x^{m}=\binom{-m}{r}x^{m+r},\    \    \
L_{0}^{(r)}\cdot x^{m}=\binom{2m}{r}x^{m}
 \end{eqnarray}
 for $r\in \BN$.
\end{example}

Set
\begin{equation*}
    e^{zL_{1}}=\sum_{n\in \BN}z^n L_{1}^{(n)}\in \CH_{+}[[z]].
\end{equation*}
We have
\begin{equation}\label{eq:eEeE-eE}
    e^{zL_{1}}e^{xL_{1}}=e^{(z+x)L_{1}}\quad\text{and}\quad e^{zL_{1}}e^{-zL_{1}}=1.
\end{equation}
Note that  we have
\begin{equation}
    z^{\deg} L_{1}^{(r)}z^{-\deg}=z^{-r}L_{1}^{(r)}
\end{equation}
on every $\BZ$-graded $\CH$-module for $r\in \BN$, which amounts to
\begin{equation}\label{eq:conj-deg-e^E}
    z^{\deg} e^{xL_{1}}z^{-\deg}=e^{z^{-1}xL_{1}}.
\end{equation}

By \eqref{FHL-1} and \eqref{FHL-2}, we have:

\begin{lemma}
Let $W$ be a $\BZ$-graded weight $\CH$-module. Then the following relations hold on $W$:
\begin{align}
    e^{zL_{1}}e^{z_0L_{-1}}&=e^{(1-zz_0)^{-1}z_0L_{-1}}(1-zz_0)^{-2\deg}e^{(1-zz_0)^{-1}zL_{1}}\nonumber\\
    &=e^{(1-zz_0)^{-1}z_0L_{-1}}e^{z(1-zz_0)L_{1}}(1-zz_0)^{-2\deg},\  \  \  \  \  \   \label{eq:ED-deg-2}\\
   (1-zz_0)^{-2\deg}&= e^{-z(1-zz_0)L_{1}}e^{-(1-zz_0)^{-1}z_0L_{-1}}e^{zL_{1}}e^{z_0L_{-1}}. \label{eq:ED-deg-3}
\end{align}
\end{lemma}

\begin{remark}
Recall that for a general bialgebra $B$, a {\em $B$-module algebra} is an associative unital algebra $A$
which is also a $B$-module satisfying the condition
\begin{equation*}
b\cdot 1_{A}=\varepsilon(b)1_{A},\    \    \   \    b(aa')=\sum (b^{(1)}a)(b^{(2)}a')
\end{equation*}
 for $b\in B,\  a,a'\in A$, where $\Delta(b)=\sum b^{(1)}\otimes b^{(2)}$ in the Sweedler notation.
\end{remark}

The following is a vertex algebra analogue
(recalling that $\B$ $(\simeq \CH_{-})$ is a subalgebra of $\CH$):

\begin{definition}\label{def-good}
An \emph{$\CH$-module vertex algebra} is a $\BZ$-graded vertex algebra $V$ which is also
a $\BZ$-graded weight $\CH$-module (with the natural $\B$-module action)
satisfying the following conditions:
\begin{enumerate}[(i)]
\item $V_n=0$ for $n$ sufficiently negative.
\label{condition-grading-truncated}
\item $L_{1}^{(n)}\1=\varepsilon(L_{1}^{(n)})\1=\delta_{n,0}\1$ for $n\in\BN$. \label{condition-delta1}
\item  For $v\in V$,
\begin{equation}
    e^{zL_{1}}Y(v,z_0)e^{-zL_{1}}=Y\bigl(e^{z(1-zz_0)L_{1}}(1-zz_0)^{-2\deg}v,z_0/(1-zz_0)\bigr). \label{condition-conj-E}
\end{equation}
\end{enumerate}
\end{definition}

\begin{remark}
Let $V$ be an $\CH$-module vertex algebra.  As $V$ is  a $\BZ$-graded vertex algebra, we have
$L_{-1}^{(n)}\1=\delta_{n,0}\1$, $L_{0}^{(n)}\1=\delta_{n,0}\1$ for $n\in \BN$.
Then $b\1=\varepsilon(b)\1$ for $b\in \{ L_{1}^{(n)},\  L_{-1}^{(n)},\  L_{0}^{(n)}\ |\ n\in \BN\}$.
Since $\CH$ as an algebra is generated by this subset, it follows that $b\1=\varepsilon(b)\1$ for all $b\in \CH$.
\end{remark}

Let $V$ be an $\CH$-module vertex algebra and let $W=\bigoplus_{n\in\BZ} W_n$ be a $\BZ$-graded $V$-module
which is {\em lower truncated} in the sense that $W_n=0$ for $n$ sufficiently negative.
Just as in~\cite{FHL}, set
\begin{equation*}
    W'=\bigoplus_{n\in\BZ} W^*_n,
\end{equation*}
which is called the \emph{restricted dual} of $W$.
Define a linear map
\begin{align*}
    Y'(v,z):V&\to (\End W')[[z,z^{-1}]]\\
    v &\mapsto Y'(v,z)=\sum_{n\in\BZ}v'_n z^{-n-1}
\end{align*}
by
\begin{equation}
    \langle Y'(v,z)w',w\rangle=\langle w',Y(e^{zL_{1}}(-z^{-2})^{\deg}v,z^{-1})w\rangle
\end{equation}
for $v\in V$, $w'\in W'$ and $w\in W$.

The following results  of \cite{FHL} (Propositions 5.3.1 and 5.3.2) are also valid here (with the same proof):

\begin{proposition}\label{th:W'-is-graded-module}
Let $V$ be an $\CH$-module vertex algebra and let $W$ be a lower truncated $\BZ$-graded module
for $V$ as a vertex algebra. Then $W'$ is a lower truncated $\BZ$-graded $V$-module.
\end{proposition}


\begin{proposition}
Let $V$ be an $\CH$-module vertex algebra and let $W$ be a lower truncated $\BZ$-graded $V$-module.
 Then $W$ is a $V$-submodule of $(W')'$.
Furthermore, if $\dim W_n<\infty$ for all $n\in\BZ$, then $W\cong (W')'$.
In this case, $W$ is irreducible if and only if $W'$ is irreducible.
\end{proposition}

On the other hand, we have (\cite[Proposition~2.8]{Li94}; cf. \cite{FHL}):

\begin{lemma}
Let $V$ be an $\CH$-module vertex algebra and
let $W$ be a lower truncated irreducible $\BZ$-graded $V$-module such that
 $\dim W_m<\infty$ for all $m\in\BZ$.
Then any nondegenerate bilinear form $(\cdot,\cdot)$ on $W$, satisfying the condition that
\begin{gather}
    (W_m,W_n)=0\quad\text{for  }m,n\in\BZ   \mbox{ with }m\ne n,  \label{eq:nondegenerate-module-1}\\
    (Y_W(v,z)w_1,w_2)=(w_1,Y_W(e^{zL_{1}}(-z^{-2})^{\deg}v,z^{-1})w_2) \quad\text{for }v\in V,\ w_1,w_2\in W,\label{eq:nondegenerate-module-2}
\end{gather}
is either symmetric or skew-symmetric.
\end{lemma}

\begin{definition}\label{def-H-module}
Let $V$ be an $\CH$-module vertex algebra.
A \emph{$(V,\CH)$-module}  is a $\BZ$-graded weight $\CH$-module $W=\bigoplus_{n\in\BZ} W_n$
which is also a $\BZ$-graded $V$-module, satisfying
\begin{equation}
e^{zL_{-1}}Y_W(v,x)e^{-zL_{-1}}=Y_{W}(e^{zL_{-1}}v,x),
\end{equation}
\begin{equation}
    e^{zL_{1}}Y_W(v,z_0)e^{-zL_{1}}=Y_W\bigl(e^{z(1-zz_0)L_{1}}(1-zz_0)^{-2\deg}v,z_0/(1-zz_0)\bigr) \label{H-module-conj-E}
\end{equation}
for $v\in V$.
\end{definition}

It is clear that an $\CH$-module vertex algebra $V$ itself is a $(V,\CH)$-module.

\begin{lemma}\label{th:Hom(VV')=deg0}
Let $V$ be an $\CH$-module vertex algebra and  let $W$ be a $(V,\CH)$-module. Then
for every $f\in \Hom_{(V,\CH)}(V,W)$, we have $f(V_n)\subset W_n$ for $n\in\BZ$.
\end{lemma}

\begin{proof}  Suppose $f\in \Hom_{(V,\CH)}(V,W)$. Let $v \in V_{n}$ with $n\in \BZ$.
Using \eqref{eq:ED-deg-3}  we get
\begin{align*}
 (1-zz_0)^{-2\deg}f(v)&= e^{-z(1-zz_0)L_{1}}e^{-(1-zz_0)^{-1}z_0L_{-1}}e^{zL_{1}}e^{z_0L_{-1}}f(v) \\
    &= f(e^{-z(1-zz_0)L_{1}}e^{-(1-zz_0)^{-1}z_0L_{-1}}e^{zL_{1}}e^{z_0L_{-1}}v) \\
  &= f((1-zz_0)^{-2\deg}v)\\
    &=  (1-zz_0)^{-2n}f(v),
\end{align*}
which implies $f(v)\in W_{n}$ by Remark \ref{1-x=deg}.
Thus $f(V_n)\subset W_n$ for all $n\in\BZ$.
\end{proof}

In the following, we shall prove that if $W$ is a lower truncated $(V,\CH)$-module, then $W'$ is a $(V,\CH)$-module.
Recall from Remark \ref{anti-automorphism} the anti-automorphism  $\theta$ of $\CH$.
For any $\CH$-module $W$, we define an $\CH$-module structure on $W^{*}$ by
\begin{equation*}
\langle af,w\rangle=\langle f,\theta(a)w\rangle
\end{equation*}
for $a\in \CH,\ f\in W^{*},\ w\in W$.
We have:

\begin{lemma}\label{th:eE-on-V'}
Let $V$ be an $\CH$-module vertex algebra and let $(W,Y_{W})$ be a lower truncated $(V,\CH)$-module.
Then $W'$ is an $\CH$-submodule of $W^{*}$ and furthermore,  $W'$ is a $(V,\CH)$-module.
\end{lemma}

\begin{proof} Since $W$ is a lower truncated $\BZ$-graded $V$-module,
it follows from Proposition~\ref{th:W'-is-graded-module} that $W'$ is a $\BZ$-graded $V$-module.
On the other hand, as $W$ is an $\CH$-module, $W^{*}$ is an $\CH$-module.
It can be readily seen that $W'$ is an $\CH$-submodule of $W^{*}$.
In particular, $W'$ is an $\CH$-module. It is clear that
$W'$ is a $\BZ$-graded $\CH$-module.

Let $w'\in W',\  a\in V,\ w\in W$.  Using  \eqref{V,B-module-property} and \eqref{eq:ED-deg-2}, we have
\begin{align*}
    \langle Y'_{W}(e^{xL_{-1}}a,z)w',w\rangle
    &=\langle Y'_{W}(a,z+x)w',w\rangle\\
  &=\langle w',Y_{W}(e^{(z+x)L_{1}}(z+x)^{-2\deg}(-1)^{\deg}a,(z+x)^{-1})w\rangle.
\end{align*}
On the other hand, using \eqref{condition-conj-E} and \eqref{eq:conj-deg-e^E} we get
\begin{align*}
    &\phantom{=}\;\;\langle e^{xL_{-1}}Y'_{W}(a,z)e^{-xL_{-1}}w',w\rangle\\
    &=\langle w',e^{-xL_{1}}Y_{W}(e^{zL_{1}}(-z^{-2})^{\deg}a,z^{-1})e^{xL_{1}}w\rangle\\
    &=\langle w', Y_{W}(e^{-x(1+z^{-1}x)L_{1}}(1+z^{-1}x)^{-2\deg}e^{zL_{1}}(-z^{-2})^{\deg}a,z^{-1}(1+z^{-1}x)^{-1})w\rangle\\
    &=\langle w', Y_{W}(e^{-x(1+z^{-1}x)L_{1}}e^{(1+z^{-1}x)^2zL_{1}}(1+z^{-1}x)^{-2\deg}(-z^{-2})^{\deg}a,(z+x)^{-1})w\rangle\\
    &=\langle w', Y_{W}(e^{(z+x)L_{1}}(z+x)^{-2\deg}(-1)^{\deg}a,(z+x)^{-1})w\rangle.
\end{align*}
Consequently, $Y_{W}'(e^{xL_{-1}}a,z)=e^{xL_{-1}}Y_{W}'(a,z)e^{-xL_{-1}}$ on $W'$.
Using \eqref{eq:conj-deg-e^E}, we have
\begin{align*}
    &\phantom{=}\;\;\langle Y_{W}'\bigl(e^{z(1-zz_0)L_{1}}(1-zz_0)^{-2\deg}a,z_0/(1-zz_0)\bigr)w',w\rangle\\
    &=\langle w',Y_{W}\bigl(e^{(1-zz_0)^{-1}z_0L_{1}}(-(1-zz_0)^2 z_0^{-2})^{\deg}e^{z(1-zz_0)L_{1}}(1-zz_0)^{-2\deg}a,z_0^{-1}(1-zz_0)\bigr)w\rangle\\
    &=\langle w',Y_{W}\bigl(e^{(1-zz_0)^{-1}z_0L_{1}}e^{-(1-zz_0)^{-1} z_0^2zL_{1}}(-(1-zz_0)^2 z_0^{-2})^{\deg}(1-zz_0)^{-2\deg}a,z_0^{-1}-z)\bigr)w\rangle\\
    &=\langle w',Y_{W}\bigl(e^{z_0L_{1}}(-z_0^{-2})^{\deg}a,z_0^{-1}-z\bigr)w\rangle.
\end{align*}
On the other hand, using \eqref{V,B-module-property}, we get
\begin{align*}
    \langle e^{zL_{1}}Y'_{W}(a,z_0)e^{-zL_{1}}w',w\rangle
    &=\langle w',e^{-zL_{-1}}Y_{W}(e^{z_0L_{1}}(-z_0^{-2})^{\deg}a,z_0^{-1})e^{zL_{-1}}w\rangle\\
    &=\langle w', Y_{W}(e^{z_0L_{1}}(-z_0^{-2})^{\deg}a,z_0^{-1}-z) w\rangle.
\end{align*}
Thus \eqref{H-module-conj-E} holds.
Therefore, $W'$ is a $(V,\CH)$-module.
\end{proof}

\section{Invariant bilinear forms on $\CH$-module vertex algebras}

In this section, we study invariant bilinear forms on $\CH$-module vertex algebras.
The main goal is to establish suitable modular analogues of the main results in characteristic $0$ (see \cite{Li94}).

\begin{definition}\label{defi-invarant-bilinear-form}
Let $V$ be an $\CH$-module vertex algebra and let $W$ be a lower truncated $(V,\CH)$-module.
A bilinear form $(\cdot,\cdot)$ on $W$  is said to be \emph{invariant} if
\begin{gather}
(Y(v,z)w,w')=(w,Y(e^{zL_{1}}(-z^{-2})^{\deg}v,z^{-1})w'), \label{condition-BF-invariant}\\
(aw,w')=(w,\theta(a)w')\label{condition-module-H-invariant}
\end{gather}
for $a\in \CH,\ v\in V,\ w,w'\in W$.
\end{definition}

\begin{remark}
Note that since $\CH$ as an algebra is generated by $L_{1}^{(n)}$ and $L_{-1}^{(n)}$ for $n\in \BN$,
 the condition (\ref{condition-module-H-invariant})  in Definition \ref{defi-invarant-bilinear-form}
is equivalent to the following condition:
\begin{gather}
(e^{xL_{\pm1}}w,w')=(w,e^{xL_{\mp 1}}w')\label{condition-module-DE-invariant}
\end{gather}
for $w,w'\in W$.
\end{remark}

\begin{lemma}\label{form-on-module}
Let $V$ be an $\CH$-module vertex algebra and let $W$ be a lower truncated $(V,\CH)$-module.
Assume that $(\cdot,\cdot)$ is an invariant bilinear form on $W$. Then
$ (W_m,W_n)=0$ for $m,n\in \BZ$ with $m\ne n$.
\end{lemma}

\begin{proof} Suppose that $(u,v)\ne 0$ for some $u\in W_{m},\ v\in W_{n}$ with $m,n\in \BZ$.
Using invariance, we have
\begin{align*}
 (1-x)^{-2\deg u}(u,v)=  ((1-x)^{-2\deg}u,v)=(u,(1-x)^{-2\deg}v)= (1-x)^{-2\deg v}(u,v),
\end{align*}
recalling (\ref{deg=L0}).
As $(u,v)\ne 0$, we have  $(1-x)^{-2m}=(1-x)^{-2n}$ in $\BF((x))$, which implies  $m=n$ in $\BZ$.
Therefore, we have $ (V_m,V_n)=0$ for $m,n\in \BZ$ with $m\ne n$.
\end{proof}

Proposition 5.3.6 of \cite{FHL} is still valid here (which follows from the same proof).

\begin{proposition}
Every invariant bilinear form on an $\CH$-module vertex algebra is symmetric.
\end{proposition}

We also have the following slightly different result (cf. Lemma \ref{form-on-module}):

\begin{lemma}
Let $V$ be an $\CH$-module vertex algebra such that
$V_n=0$ for $n<0$. Suppose that $(\cdot,\cdot)$ is a bilinear form on $V$ such that
invariance \eqref{condition-BF-invariant} holds
and $(V_m,V_n)=0$ for $m,n\in \BZ$ with $m\ne n$. Then $(\cdot,\cdot)$ is symmetric.
\end{lemma}

\begin{proof}
Let $a\in V_0$. Since $L_{1}^{(i)}V_n\subset V_{n-i}$ for $i\in\BN,\  n\in\BZ$,
we have $L_{1}^{(i)}a=0$ for $i\ge 1$.
Hence
\begin{align*}
    (Y(a,x)\1,\1)
    =(\1,Y(e^{xL_{1}}(-x^{-2})^{\deg}a,x^{-1})\1)
    =(\1,Y(a,x^{-1})\1).
\end{align*}
Extracting the constant terms from both sides we get  $(a,\1)=(\1,a)$.
As $(\1,V_{n})=0=(V_{n},\1)$ for $n\ne 0$, we get $(v,\1)=(\1,v)$ for all $v\in V$.
Therefore, for $u,v\in V$, we have
\begin{align*}
    (u,v)
    &=\Res_x x^{-1} ( Y( u,x)\1,v)\\
    &=\Res_x x^{-1} (\1,Y(e^{xL_{1}}(-x^{-2})^{\deg} u,x^{-1})v)\\
    &=\Res_x x^{-1} (Y(e^{xL_{1}}(-x^{-2})^{\deg} u,x^{-1})v,\1)\\
    &=\Res_x x^{-1} (v,Y(u,x)\1)\\
    &=(v,u),
\end{align*}
as desired.
\end{proof}

Notice that if $V=\bigoplus_{n\in \BZ}V_{n}$ is a $\BZ$-graded vector space, then a bilinear form $(\cdot,\cdot)$ on $V$ such that
$(V_{m},V_{n})=0$ for $m,n\in \BZ$ with $m\ne n$ amounts to a grading-preserving linear map from $V$ to $V'$.
With this identification, we have the following result (cf. \cite{Li94}):

\begin{proposition}\label{th:space-BF=Hom(VV')}
Let $V$ be an $\CH$-module vertex algebra.
Then the space of invariant bilinear forms on $V$ is canonically isomorphic to $\Hom_{(V,\CH)}(V,V')$.
\end{proposition}

\begin{proof}
Let $B(\cdot,\cdot)$ be a bilinear form on $V$. Define $\Psi_{B}:V\to V'$ by $\Psi_{B}(v)(u)=B(v,u)$ for $u,v\in V$.
Then $B(\cdot,\cdot)$ satisfies invariance (\ref{condition-BF-invariant})  if and only if $\Psi_{B}\in \Hom_V(V,V')$.
For $a\in \CH,\ u,v\in V$, we have
\begin{align*}
    \Psi_{B}(av)(u)=B(av,u)=B(v,\theta(a)u)=\langle \Psi_{B}(v),\theta(a)u\rangle
  =\langle a\Psi_{B}(v),u\rangle=  a\Psi_{B}(v)(u).
\end{align*}
This proves $\Psi_{B}\in\Hom_{(V,\CH)}(V,V')$.

On the other hand, let $f\in \Hom_{(V,\CH)}(V,V')$.
Define a bilinear form $(\cdot,\cdot)_{f}$ on $V$ by $(u,v)_f=\langle f(u),v\rangle$ for $u,v\in V$.
Then
\begin{align*}
    (Y(u,z)v,w)_f&=\langle fY(u,z)v,w\rangle=\langle Y(u,z)f(v),w\rangle
    =\langle f(v),Y(e^{zL_{1}}(-z^{-2})^{\deg}u,z^{-1})w\rangle\\
    &=(v,Y(e^{zL_{1}}(-z^{-2})^{\deg}u,z^{-1})w)_f.
\end{align*}
Furthermore,  we have
\begin{gather*}
    (au,v)_f=\langle f(au),v\rangle=\langle af(u),v\rangle
    =\langle f(u),\theta(a)v\rangle=(u,\theta(a)v)_f
    \end{gather*}
for $a\in \CH$.
Therefore $(\cdot,\cdot)_f$ is an invariant bilinear form. This concludes the proof.
\end{proof}

Furthermore, we have:

\begin{proposition}\label{WH-invariant}
Let $V$ be an $\CH$-module vertex algebra and let $W$ be a $(V,\CH)$-module.
Set
\begin{eqnarray}
W^{\CH}=\{ w\in W \ | \  L_{1}^{(n)}w=0=L_{-1}^{(n)}w\   \  \mbox{ for }n\ge 1\}.
\end{eqnarray}
Then  the assignment $f\mapsto f(\1)$
 is a linear isomorphism from $\Hom_{(V,\CH)}(V,W)$ onto $W^\CH$. Furthermore,
 we have $W^\CH\subset W_0$.
\end{proposition}

\begin{proof} It is clear that $\1\in V^{\CH}$. Then
 $f(\1)\in W^\CH$ for $f\in \Hom_{(V,\CH)}(V,W)$.
 Thus the assignment $f\mapsto f(\1)$ is a map
 from $\Hom_{(V,\CH)}(V,W)$ to $W^\CH$, which is denoted by $\phi$.
 Clearly, $\phi$ is linear.
If $f\in\ker \phi$, then $f(v)=f(v_{-1}\1)=v_{-1}f(\1)=v_{-1}\phi(f)=0$ for every $v\in V$,
which implies $f=0$. Thus $\phi$ is injective.
On the other hand, let $w\in W^\CH$.
By Lemma~\ref{vacuum-like},
we have a $V$-module homomorphism $f$ from $V$ to $W$ with $f(v)=v_{-1}w$ for $v\in V$.
In particular, $f(\1)=w$. Furthermore, by Lemma~ \ref{th:f(1)-in-W^D},
\begin{equation*}
f(e^{xL_{-1}}v)=e^{xL_{-1}}f(v)\   \   \  \mbox{ for }v\in V.
\end{equation*}
On the other hand, for $v\in V$, we have
\begin{align*}
    f(e^{xL_{1}}v)&=\Res_{z}z^{-1} f(e^{xL_{1}}Y(v,z)\1)=\Res_{z}z^{-1} f(e^{xL_{1}}Y(v,z)e^{-xL_{1}}\1)\\
    &=\Res_{z}z^{-1} f(Y(e^{x(1-xz)L_{1}}(1-xz)^{-2\deg}v,(1-xz)^{-1}z)\1)\\
    &=\Res_{z}z^{-1} Y_{W}(e^{x(1-xz)L_{1}}(1-xz)^{-2\deg}v,(1-xz)^{-1}z)f(\1)\\
    &=\Res_{z}z^{-1} e^{xL_{1}}Y_{W}(v,z)e^{-xL_{1}}w\\
    &=\Res_{z}z^{-1} e^{xL_{1}}Y_{W}(v,z)w\\
    &=e^{xL_{1}}v_{-1}w\\
    &=e^{xL_{1}}f(v).
\end{align*}
Thus  $f\in\Hom_{(V,\CH)}(V,W)$.
This proves that $\phi$ is surjective. Therefore, $\phi$ is a linear isomorphism.

For the second assertion, let $w\in W^\CH$. Using  \eqref{eq:ED-deg-3} we have
\begin{align*}
    (1-zz_0)^{-2\deg}w=e^{-z(1-zz_0)L_{1}}e^{-(1-zz_0)^{-1}z_0L_{-1}}e^{zL_{1}}e^{z_0L_{-1}}w=w,
\end{align*}
which implies $-2(\deg w)=0$. That is, $w\in W_{0}$. Thus $W^{\CH}\subset W_{0}$.
\end{proof}

\begin{remark}\label{rem:W^H}
Let $V$ be an $\CH$-module vertex algebra and let $W$ be a $(V,\CH)$-module.
For $w\in W$,  from definition, $w\in W^\CH$ if and only if $L_{\pm 1}^{(n)}w=\varepsilon(L_{\pm 1}^{(n)})w$
for $n\in \BN$. (Recall that $\varepsilon$ is the counit map of $\CH$.)
As $\CH$ is generated by $L_{\pm 1}^{(n)}$ for $n\in \BN$, we have
$aw=\varepsilon(a)w$ for $a\in \CH$. Thus
\begin{equation*}
W^\CH=\{ w\in W\ |\ aw=\varepsilon(a)w\   \text{ for }a\in \CH\}.
\end{equation*}
\end{remark}

Let $V$ be an $\CH$-module vertex algebra. Set
\begin{equation}
  L_{1}^{+} V=\sum_{n\ge 1}L_{1}^{(n)}V,\qquad L_{-1}^{+}V=\sum_{n\ge1}L_{-1}^{(n)}V.
\end{equation}
Since $L_{1}^{(n)}$ and $L_{-1}^{(n)}$ for $n\in\BZ_+$ are homogeneous operators,
$L_{1}^{+} V$ and $L_{-1}^{+}V$ are $\BZ$-graded subspaces of $V$.
We have $(L_{1}^{(n)}V)_{0}=L_1^{(n)}V_{n}$ and $(L_{-1}^{(n)}V)_{0}=L_{-1}^{(n)}V_{-n}$
for $n\in \BZ_{+}$.

Here, we have:

\begin{lemma}\label{D+inE+V}
Let $V$ be an $\CH$-module vertex algebra.
Then $(L_{-1}^{+}V)_0\subset (L_{1}^{+}V)_0$.
\end{lemma}

\begin{proof}  First, we use induction on positive integer $t$ to show that
\begin{align}\label{eq:DV-subset-EV-03}
    L_{-1}^{(kp^t+i)}V_{-kp^t-i}\subset L_{-1}^{(kp^t+p^t)}V_{-kp^t-p^t}+(L_{1}^{+}V)_0
\end{align}
for $k\in\BN,\ 1\le  i\le p^t.$
 Recall that
\begin{align*}
    L_{1}^{(m)}L_{-1}^{(n)}&=\sum_{i=0}^{\min(m,n)}(-1)^iL_{-1}^{(n-i)}\binom{-2\deg-m-n+2i}{i}L_{1}^{(m-i)}
\end{align*}
for $m,n\in \BN$. In particular, we have
\begin{align*}
    L_{1}^{(1)}L_{-1}^{(n+1)}v
    =L_{-1}^{(n+1)}L_{1}^{(1)}v-nL_{-1}^{(n)}v
\end{align*}
for $v\in V_{-n}$ with $n\in \BN$. From this we get
\begin{align}
    L_{-1}^{(n)}V_{-n}\subset L_{-1}^{(n+1)}V_{-n-1}+L_{1}^{(1)}V_1
\end{align}
if $p\nmid n,$ noticing that $L_{-1}^{(n+1)}V_{-n}\subset V_1$.
Using this relation (repeatedly), we obtain
\begin{align}\label{eq:DV-subset-EV-01}
    L_{-1}^{(kp+i)}V_{-kp-i}\subset L_{-1}^{(kp+p)}V_{-kp-p}+L_{1}^{(1)}V_1
\end{align}
 for $k\in\BN,\  1\le  i\le p.$
This shows that \eqref{eq:DV-subset-EV-03} holds for $t=1$.

Now, assume that $t\ge 2$ and  \eqref{eq:DV-subset-EV-03} holds. 
For any $k\in \BN$ and for any $v\in V_{-kp^t}$, we have
\begin{align*}
    &L_{1}^{(p^t)}L_{-1}^{(kp^t+p^t)}v=\sum_{i=0}^{p^t}(-1)^iL_{-1}^{(kp^t+p^t-i)}\binom{-2\deg-kp^t-2p^t+2i}{i}L_{1}^{(p^t-i)}v\\
    &=-\binom{kp^t}{p^t}L_{-1}^{(kp^t)}v+\sum_{i=0}^{p^t-1}(-1)^iL_{-1}^{(kp^t+p^t-i)}\binom{-2\deg-(k+2)p^t+2i}{i}L_{1}^{(p^t-i)}v\\
    &=-\binom{kp^t}{p^t}L_{-1}^{(kp^t)}v+\sum_{i=1}^{p^t}(-1)^{p^t-i}L_{-1}^{(kp^t+i)}\binom{-2\deg-kp^t-2i}{p^t-i}L_{1}^{(i)}v.
\end{align*}
Note that by Lucas' theorem, $\binom{kp^t}{p^t}\not\equiv 0\pmod{p}$ if $p\nmid k$.
Using this and induction hypothesis \eqref{eq:DV-subset-EV-03}, we get
\begin{align}\label{eq:haha01}
    L_{-1}^{(kp^t)}V_{-kp^t}\subset \sum_{i=1}^{p^t}L_{-1}^{(kp^t+i)}V_{-kp^t-i}+(L_{1}^{+}V)_0
    \subset L_{-1}^{((k+1)p^t)}V_{-(k+1)p^t}+(L_{1}^{+}V)_0
\end{align}
for any $k\in \BN$ with $p\nmid k$.

Now, let $k\in \BN$ and $1\le i\le p^{t+1}-1$.
Then $i=i'+jp^{t}$ for some integers $i'$ and $j$ such that $1\le i'\le p^{t}$, $0\le j\le p-1$.
Using induction hypothesis \eqref{eq:DV-subset-EV-03} and using \eqref{eq:haha01} repeatedly,
we get
\begin{eqnarray*}
  && L_{-1}^{(kp^{t+1}+i)}V_{-kp^{t+1}-i}
    =L_{-1}^{((kp+j)p^t+i')}V_{-(kp+j)p^t-i'}
    \subset L_{-1}^{((kp+j+1)p^t)}V_{-(kp+j+1)p^t}+(L_{1}^{+}V)_0\\
   &&\hspace{1cm} \subset L_{-1}^{(kp^{t+1}+p^{t+1})}V_{-kp^{t+1}-p^{t+1}}+(L_{1}^{+}V)_0.
\end{eqnarray*}
This completes the induction, proving \eqref{eq:DV-subset-EV-03}.

Notice that \eqref{eq:DV-subset-EV-03} implies that for any $n\ge 1$, $L_{-1}^{(n)}V_{-n}\subset  L_{-1}^{(p^t)}V_{-p^t}+(L_{1}^{+}V)_0$
for any positive integer $t$ such that $n\le p^t$.
As $V$ is locally truncated, there exists a positive integer $t$ such that $V_{-m}=0$ for all $m\ge p^t$.
Then we get
\begin{align*}
    (L_{-1}^{+}V)_0=\sum_{n\ge 1}L_{-1}^{(n)}V_{-n}=\sum_{n=1}^{p^t}L_{-1}^{(n)}V_{-n}
    \subset L_{-1}^{(p^t)}V_{-p^t}+(L_{1}^{+}V)_0 =(L_{1}^{+}V)_0,
\end{align*}
as desired.
\end{proof}

As the main result of this section, we have (cf. \cite{Li94}):

\begin{theorem}\label{identification}
Let $V$ be an  $\CH$-module vertex algebra. Then the space of all (symmetric) invariant bilinear forms on $V$
is linearly isomorphic to the dual space of $V_0/(L_{1}^{+} V)_0$.
\end{theorem}

\begin{proof} In view of Propositions \ref{th:space-BF=Hom(VV')} and \ref{WH-invariant},
the space of all symmetric invariant bilinear forms on $V$
is linearly isomorphic to the subspace $(V')^{\CH}$ of $V'$. It then remains to show that $(V')^{\CH}$ is
linearly isomorphic to the dual space of $V_0/(L_{1}^{+} V)_0$.

Let $\Phi$ denote the composition of the natural quotient map of $V_0$ onto $V_0/(L_{1}^{+} V)_0$
with the projection map of $V$ onto $V_0$:
\begin{equation*}
\Phi: \   V\to V_0\to V_0/(L_{1}^{+} V)_0.
\end{equation*}
The dual $\Phi^{*}$ of $\Phi$ is an injective map from $(V_0/(L_{1}^{+} V)_0)^{*}$ to $V'\subset V^{*}$.
We claim
\begin{eqnarray}
(V')^\CH =\Phi^{*}( (V_0/(L_{1}^{+} V)_0)^*).
\end{eqnarray}
Let $v'\in (V')^{\CH}$. Then $L_{-1}^{(n)}v'=0$ for $n\ge 1$ and $v'\in (V')_{0}$ by Lemma \ref{WH-invariant}.
Hence, $\langle v',V_{m}\rangle =0$ for $m\ne 0$ and
$\langle v',L_{1}^{(n)}V\rangle=0$ for $n\ge 1$.
Thus $v'\in\Phi^{*}( (V_0/(L_{1}^{+} V)_0)^*)$.

On the other hand, let $f\in(V_0/(L_{1}^{+} V)_0)^*$.  Then we have $\langle \Phi^{*}(f),V_{m}\rangle=0$
for $m\ne 0$ and $\langle \Phi^{*}(f),(L_{1}^{+} V)_{0}\rangle=0$, which imply
 that $\langle \Phi^{*}(f),L_{1}^{+} V\rangle=0$.
In view of Lemma \ref{D+inE+V}, we have  $\langle \Phi^{*}(f),L_{1}^{+} V+L_{-1}^{+}V\rangle=0$.
It follows that $\langle L_{-1}^{(n)}\Phi^{*}(f),v\rangle =\langle  \Phi^{*}(f), L_{1}^{(n)}v\rangle=0$
and $\langle L_{1}^{(n)}\Phi^{*}(f),v\rangle=\langle \Phi^{*}(f),L_{-1}^{(n)}v\rangle=0$
for $n\ge 1$ and $v\in V$. Thus $L_{-1}^{(n)}\Phi^{*}(f)=0=L_{1}^{(n)}\Phi^{*}(f)$ for $n\ge 1$.
This proves $\Phi^{*}(f)\in (V')^{\CH}$.
\end{proof}

Notice that if $(\cdot,\cdot)$ is a (symmetric) invariant bilinear form on an
$\CH$-module vertex algebra $V$, then the kernel of $(\cdot,\cdot)$ is
a $(V,\CH)$-submodule of $V$, in particular, an ideal of $V$.
As an immediate consequence of Theorem \ref{identification}, we have:

\begin{corollary}\label{existence-form}
Suppose that $V$ is an  $\CH$-module vertex algebra such that $V_{n}=0$ for $n<0$ and $V_0=\BF \1$
and such that $L_{1}^{(n)}V_{n}=0$  for $n\ge 1$.
Then the space of  invariant bilinear forms  on $V$ is one-dimensional and
 there exists a symmetric invariant bilinear form $(\cdot,\cdot)$ on $V$ with  $(\1,\1)=1$.
Furthermore, if $V$ is simple,
there exists a non-degenerate symmetric invariant bilinear form $(\cdot,\cdot)$ on $V$,
which is uniquely determined by  $(\1,\1)=1$.
\end{corollary}

Assume that $V=\bigoplus_{n\in \BZ}V_{n}$ is an $\CH$-module vertex algebra such that
$V_{n}=0$ for $n<0$ and $V_0=\BF \1$.
Notice that for every proper graded $V$-submodule $U$ of $V$, we have $U\cap V_0=0$.
Let $J$ be the sum of all proper graded $V$-submodules of $V$. Then $J$ is the (unique) maximal proper graded
$V$-submodule of $V$ and we have $J\cap V_0=0$.

\begin{lemma}\label{maximal-ideal}
Suppose that $V=\bigoplus_{n\in \BZ}V_{n}$ is an $\CH$-module vertex algebra such that
$V_{n}=0$ for $n<0$ and $V_0=\BF \1$. Then the maximal proper graded
$V$-submodule $J$ of $V$ is an ideal of $V$. If in addition we assume
$L_{1}^{(n)}V_{n}=0$ for all $n\ge 1$, then $J$ is also an $\CH$-submodule of $V$.
\end{lemma}

\begin{proof} Denote by $L_{-1} J$ the linear span of $L_{-1}^{(n)}J$ for $n\in \BN$.
As $L_{-1}^{(n)}V_{m}\subset V_{m+n}$ for $n,m\in \BN$, we have $L_{-1} J\cap V_{0}=0$.
Note that for $v\in V,\ w\in J$, we have
\begin{equation*}
Y(v,x)e^{zL_{-1}}w=e^{zL_{-1}}Y(e^{-zL_{-1}}v,x)w.
\end{equation*}
It follows that $L_{-1} J$ is a graded $V$-submodule of $V$.
From the definition of $J$,  we have $L_{-1} J\subset J$. Then $J$ is an ideal of $V$.
Furthermore, assume $L_{1}^{(n)}V_{n}=0$ for all $n\ge 1$.
Let $L_{1} J$ be the linear span of $L_{1}^{(n)}J$ for $n\in \BN$. It is clear that $L_{1} J$ is graded
with $L_{1} J\cap V_0=0$.
As
\begin{equation*}
Y(v,x)e^{zL_{1}}w=e^{zL_{1}}Y\left(e^{-z(1+zx)L_{1}}(1+zx)^{-2\deg}v,\frac{x}{1+zx}\right)w
\end{equation*}
for $v\in V,\ w\in J$,  $L_{1} J$ is a $V$-submodule of $V$.
Thus $L_{1} J\subset J$. This shows that $J$ is an $\CH$-submodule of $V$.
Consequently, $V/J$ is an $\CH$-module vertex algebra and it is a simple graded $V$-module.
\end{proof}

At the end of this section, we prove a technical result which will be used in the next section.

\begin{lemma}\label{th:induction}
Let $V$ be a $\BZ$-graded vertex algebra which is also an $\CH$-module such that
\begin{equation*}
L_{-1}^{(n)}=\D^{(n)},\    \    \    \    L_{0}^{(n)}=\binom{-2\deg}{n}
\end{equation*}
 on $V$ for $n\in \BN$.  Assume
\begin{gather}\label{eq:conj-E}
    e^{zL_{1}}Y(a,z_0)e^{-zL_{1}}=Y\bigl(e^{z(1-zz_0)L_{1}}(1-zz_0)^{-2\deg}a,z_0/(1-zz_0)\bigr)
\end{gather}
for all $a\in T$, where $T$ is a generating subset of $V$. Then $V$ is an $\CH$-module vertex algebra.
\end{lemma}

\begin{proof} Let $U$ consist of vectors $a\in V$ such that \eqref{eq:conj-E} holds.
It suffices to show $V=U$. It is clear that $\1\in U$ and we have $T\subset U$ from assumption.
Then it suffices to prove that $U$ is a vertex subalgebra.
Set
\begin{equation*}
    \Phi(x,z)=e^{x(1-xz)L_{1}}(1-xz)^{-2\deg}.
\end{equation*}
For $a\in U$, as $z^{\deg}Y(a,z_0)z^{-\deg}=Y(z^{\deg }a,zz_0)$, we have that
\begin{eqnarray}\label{eq:new-add}
    \Phi(x,x_2)Y(a,x_0)=Y\left(\Phi(x,x_2+x_0)a, \frac{x_0}{(1-xx_2)(1-x(x_2+x_0))} \right)\Phi(x,x_2).
\end{eqnarray}
Let $u,v\in U$. Using Jacobi identity, \eqref{eq:eEeE-eE}, \eqref{eq:conj-E} and $\delta$-function properties,
we have
\begin{align*}
    &\quad~ x_2^{-1}\delta\biggl(\frac{x_1-x_0}{x_2}\biggr)e^{xL_{1}}Y(Y(u,x_0)v,x_2)e^{-xL_{1}}\\
    &=x_0^{-1}\delta\biggl(\frac{x_1-x_2}{x_0}\biggr)e^{xL_{1}}Y(u,x_1)Y(v,x_2)e^{-xL_{1}}
        -x_0^{-1}\delta\biggl(\frac{x_2-x_1}{-x_0}\biggr)e^{xL_{1}}Y(v,x_2)Y(u,x_1)e^{-xL_{1}}\\
    &=x_0^{-1}\delta\biggl(\frac{x_1-x_2}{x_0}\biggr)Y\bigl(\Phi(x,x_1)u,x_1/(1-xx_1)\bigr)
       Y\bigl(\Phi(x,x_2)v,x_2/(1-xx_2)\bigr)\\
        &\qquad    -x_0^{-1}\delta\biggl(\frac{x_2-x_1}{-x_0}\biggr)
        Y\bigl(\Phi(x,x_2)v,x_2/(1-xx_2)\bigr)Y\bigl(\Phi(x,x_1)u,x_1/(1-xx_1)\bigr)\\
    &=x_0^{-1}\delta\biggl(\frac{\frac{x_1}{1-xx_1}-\frac{x_2}{1-xx_2}}{\frac{x_0}{(1-xx_1)(1-xx_2)}}\biggr)
    Y\bigl(\Phi(x,x_1)u,x_1/(1-xx_1)\bigr)Y\bigl(\Phi(x,x_2)v,x_2/(1-xx_2)\bigr)\\
        &\qquad    -x_0^{-1}\delta\biggl(\frac{\frac{x_2}{1-xx_2}-\frac{x_1}{1-xx_1}}{-\frac{x_0}{(1-xx_1)(1-xx_2)}}\biggr)
        Y\bigl(\Phi(x,x_2)v,x_2/(1-xx_2)\bigr)Y\bigl(\Phi(x,x_1)u,x_1/(1-xx_1)\bigr)\\
    &=(1-xx_1)^{-1}(1-xx_2)^{-1}\biggl(\frac{x_0}{(1-xx_1)(1-xx_2)}\biggr)^{-1}\delta\biggl(\frac{\frac{x_1}{1-xx_1}-\frac{x_2}{1-xx_2}}{\frac{x_0}{(1-xx_1)(1-xx_2)}}\biggr)\times\\
        &\qquad Y\bigl(\Phi(x,x_1)u,x_1/(1-xx_1)\bigr)Y\bigl(\Phi(x,x_2)v,x_2/(1-xx_2)\bigr)\\
        &\qquad    -(1-xx_1)^{-1}(1-xx_2)^{-1}\biggl(\frac{x_0}{(1-xx_1)(1-xx_2)}\biggr)^{-1}\delta\biggl(\frac{\frac{x_2}{1-xx_2}-\frac{x_1}{1-xx_1}}{-\frac{x_0}{(1-xx_1)(1-xx_2)}}\biggr)\times\\
        &\qquad Y\bigl(\Phi(x,x_2)v,x_2/(1-xx_2)\bigr)Y\bigl(\Phi(x,x_1)u,x_1/(1-xx_1)\bigr)\\
    &=(1-xx_1)^{-1}(1-xx_2)^{-1}\frac{1-xx_2}{x_2}\delta\biggl(\frac{\frac{x_1}{1-xx_1}-\frac{x_0}{(1-xx_1)(1-xx_2)}}{\frac{x_2}{1-xx_2}}\biggr)\times\\
    &\qquad Y(Y\bigl(\Phi(x,x_1)u,x_0/(1-xx_1)(1-xx_2)\bigr)\Phi(x,x_2)v,x_2/(1-xx_2)\bigr)\\
    &=(1-xx_1)^{-1}(1-xx_2)^{-1}\Biggl(\biggl(\frac{x_0}{(1-xx_1)(1-xx_2)}\biggr)^{-1}\delta\biggl(\frac{\frac{x_1}{1-xx_1}-\frac{x_2}{1-xx_2}}{\frac{x_0}{(1-xx_1)(1-xx_2)}}\biggr)\\
        &\qquad -\biggl(\frac{x_0}{(1-xx_1)(1-xx_2)}\biggr)^{-1}\delta\biggl(\frac{\frac{x_2}{1-xx_2}-\frac{x_1}{1-xx_1}}{\frac{-x_0}{(1-xx_1)(1-xx_2)}}\biggr)\Biggr)\times  \\
        &\qquad Y(Y\bigl(\Phi(x,x_1)u,x_0/(1-xx_1)(1-xx_2)\bigr)\Phi(x,x_2)v,x_2/(1-xx_2)\bigr)\\
    &=(1-xx_1)^{-1}(1-xx_2)^{-1}\Biggl(\biggl(\frac{x_0}{(1-xx_1)(1-xx_2)}\biggr)^{-1}\delta\biggl(\frac{x_1-x_2}{x_0}\biggr)\\
        &\qquad -\biggl(\frac{x_0}{(1-xx_1)(1-xx_2)}\biggr)^{-1}\delta\biggl(\frac{x_2-x_1}{-x_0}\biggr)\Biggr)\times \\
        &\qquad Y(Y\bigl(\Phi(x,x_1)u,x_0/(1-xx_1)(1-xx_2)\bigr)\Phi(x,x_2)v,x_2/(1-xx_2)\bigr)\\
    &=\Biggl(x_0^{-1}\delta\biggl(\frac{x_1-x_2}{x_0}\biggr)-x_0^{-1}\delta\biggl(\frac{x_2-x_1}{-x_0}\biggr)\Biggr)\times \\
    &\qquad Y(Y\bigl(\Phi(x,x_1)u,x_0/(1-xx_1)(1-xx_2)\bigr)\Phi(x,x_2)v,x_2/(1-xx_2)\bigr)\\
    &=x_2^{-1}\delta\biggl(\frac{x_1-x_0}{x_2}\biggr)Y(Y\bigl(\Phi(x,x_1)
       u,x_0/(1-xx_1)(1-xx_2)\bigr)\Phi(x,x_2)v,x_2/(1-xx_2)\bigr).
\end{align*}
On the other hand, using \eqref{eq:conj-deg}, \eqref{eq:new-add} and $\delta$-function properties, we have
\begin{align*}
    &\quad~ x_2^{-1}\delta\biggl(\frac{x_1-x_0}{x_2}\biggr)
    Y\left(e^{x(1-xx_2)L_{1}}(1-xx_2)^{-2\deg}Y(u,x_0)v,\frac{x_2}{1-xx_2}\right)\\
    &=x_2^{-1}\delta\biggl(\frac{x_1-x_0}{x_2}\biggr)Y\left(\Phi(x,x_2)Y(u,x_0)v,\frac{x_2}{1-xx_2}\right)\\
    &=x_2^{-1}\delta\biggl(\frac{x_1-x_0}{x_2}\biggr)Y\left(Y\left(\Phi(x,x_2+x_0)u,\frac{x_0}{(1-xx_2)(1-x(x_2+x_0))}\right)\Phi(x,x_2)v,
    \frac{x_2}{1-xx_2}\right)\\
    &=x_2^{-1}\delta\biggl(\frac{x_1-x_0}{x_2}\biggr)Y\left(Y\left(\Phi(x,x_1)
      u,\frac{x_0}{(1-xx_2)(1-xx_1)}\right)\Phi(x,x_2)v,\frac{x_2}{1-xx_2}\right).
\end{align*}
It follows that $U$ is closed.
Therefore, $U$ is a vertex subalgebra, as desired.
\end{proof}

\section{Affine vertex algebras and Virasoro vertex algebras}

In this section, we restrict ourselves to affine vertex algebras and Virasoro vertex algebras over field $\BF$.

First, we study  affine vertex algebras.
Let $\fg$ be a Lie algebra over $\BF$ equipped with a symmetric invariant  bilinear form $\langle\cdot,\cdot\rangle$
and let $\hat{\fg}=\fg\otimes \BF[t,t^{-1}]\oplus \BF {\bf k}$ be the associated affine Lie algebra.

Recall that for a bialgebra $B$,
a {\em $B$-module Lie algebra} is a Lie algebra $\fg$ with a $B$-module structure such that
\begin{align}
b[u,v]=\sum [b^{(1)}u,b^{(2)}v]\   \   \   \text{ for }b\in B,\ u,v\in \fg.
\end{align}

\begin{lemma}
The affine Lie algebra $\hat\fg$ is an $\CH$-module Lie algebra with
\begin{eqnarray}
&& L_{-1}^{(r)} (\bk)=\delta_{r,0}\bk,\     \    \        L_{-1}^{(r)} (a\otimes t^{n})=(-1)^{r}\binom{n}{r} a\otimes t^{n-r},\\
&&L_{1}^{(r)} (\bk)=\delta_{r,0}\bk,\     \    \    L_{1}^{(r)} (a\otimes t^{n})=\binom{-n}{r} a\otimes t^{n+r},   \\
&& L_{0}^{(r)}(\bk)=\delta_{r,0}\bk,\   \   \   \    L_{0}^{(r)}(a\otimes t^{n})=\binom{2n}{r} a\otimes t^{n}
\end{eqnarray}
for $r\in\BN,\ a\in \fg,\ n\in \BZ$.
\end{lemma}

\begin{proof}  Recall from Example \ref{Laurent} that $\BF[x,x^{-1}]$ is an $\CH$-module algebra.
From this we see that $\hat{\fg}$ with the defined action becomes an $\CH$-module,
which is the direct sum of $\CH$-modules $\BF {\bf k}$ and $\fg\otimes \BF[t,t^{-1}]$.
Furthermore, for $r\in \BN$, $a,b\in\fg$, $m,n\in\BZ$, we have
\begin{align*}
    &\quad~L_{1}^{(r)}[a\otimes t^{m},b\otimes t^{n}]\\
    &=L_{1}^{(r)}([a,b]\otimes t^{m+n})+m\langle a,b\rangle\delta_{m+n,0}\delta_{r,0}\bk\\
    &=\binom{-m-n}{r}[a,b]\otimes t^{m+n+r}+m\langle a,b\rangle\delta_{m+n,0}\delta_{r,0}\bk\\
    &=\sum_{i=0}^r\binom{-m}{r-i}\binom{-n}{i}\left([a,b]\otimes t^{m+n+r}+(m+r-i)\langle a,b\rangle\delta_{m+n+r,0}\bk\right)\\
    &=\sum_{i=0}^r[L_{1}^{(r-i)}(a\otimes t^{m}),L_{1}^{(i)}(b\otimes t^{n})],
\end{align*}
noticing that for $r=0$,
\begin{equation*}
    \sum_{i=0}^r\binom{-m}{r-i}\binom{-n}{i}(m+r-i)=m
\end{equation*}
and for $r\ge 1$,
\begin{align*}
    &\phantom{=}\,\,\delta_{m+n+r,0}\sum_{i=0}^r\binom{-m}{r-i}\binom{-n}{i}(m+r-i)\\
    &=m\delta_{m+n+r,0}\sum_{i=0}^r\binom{-m}{r-i}\binom{-n}{i}+\sum_{i=0}^r\binom{-m}{r-i}\binom{-n}{i}(r-i)\\
    &=\delta_{m+n+r,0}\left(m\binom{-m-n}{r}+\sum_{i=0}^{r-1}\binom{-m}{r-i}\binom{-n}{i}(r-i)\right)\\
    &=\delta_{m+n+r,0}\left(m\binom{-m-n}{r}-m\binom{-m-n-1}{r-1}\right)\\
    &=0.
\end{align*}
It was proved in \cite{LM} that
\begin{align*}
   L_{-1}^{(r)}[a\otimes t^{m},b\otimes t^{n}]=\sum_{i=0}^r[L_{-1}^{(r-i)}(a\otimes t^{m}),L_{-1}^{(i)}(b\otimes t^{n})].
\end{align*}
Since $\CH$ as an algebra is generated by $L_{\pm 1}^{(n)}$ for $n\in \BN$,
 it follows that $\hat\fg$ is an $\CH$-module Lie algebra.
\end{proof}

For $a\in \fg,\ n\in \BZ$, as a common practice we shall alternatively use $a(n)$ for $a\otimes t^{n}$. Set
\begin{align}
a(x)=\sum_{n\in \BZ}a(n)x^{-n-1}\in \hat{\fg}[[x,x^{-1}]].
\end{align}
We have
\begin{align}
e^{zL_{-1}}(a(x))=a(x+z),\   \   \   \   e^{zL_1}(a(x))=(1-xz)^{-2}a(x(1-xz)^{-1}).
\end{align}

With $\hat\fg$ an $\CH$-module Lie algebra,
$U(\hat\fg)$ is naturally an $\CH$-module algebra (cf. \cite{LM}).
Let $\ell\in \BF$. Recall the vertex algebra $V_{\hat{\fg}}(\ell,0)$, whose underlying space is the
$\hat{\fg}$-module $U(\hat{\fg})/J_{\ell}$, where $J_{\ell}$ is the left ideal generated by
$\fg\otimes \BF[t]$ and ${\bf k}-\ell$. Furthermore, the vacuum vector $\1$  is $1+J_{\ell}$ and  we have
\begin{align}
Y(a,x)=a(x)\   \   \   \mbox{ for }a\in \fg,
\end{align}
where $\fg$ is considered as a subspace
of $V_{\hat{\fg}}(\ell,0)$ with $a$ being identified with $a(-1)\1$ for $a\in \fg$.
Notice that $\fg\otimes \BF[t]+ \BF (\bk-\ell)$ is an $\CH$-submodule of $U(\hat{\fg})$.
It follows that $J_{\ell}$ is an  $\CH$-submodule. Then
$V_{\hat{\fg}}(\ell,0)$ is naturally an $\CH$-module (with $L_{1}^{(r)}\1=\delta_{r,0}\1$ and $L_{-1}^{(r)}\1=\delta_{r,0}\1$
for $r\in \BN$).

View $\hat{\fg}$ as a $\BZ$-graded Lie algebra with
\begin{equation*}
\deg {\bf k}=0,\   \   \   \   \deg (\fg\otimes t^{n})=-n\   \  \text{ for }n\in \BZ,
\end{equation*}
so that $U(\hat{\fg})$ is a $\BZ$-graded algebra. We see that $U(\hat{\fg})$ is a $\BZ$-graded $\CH$-module.
 It is clear that $J_{\ell}$ is a graded submodule. Consequently,
$V_{\hat{\fg}}(\ell,0)$ is a $\BZ$-graded vertex algebra and a graded $\CH$-module.
Furthermore, we have:

\begin{proposition}\label{affine-module-va}
Vertex algebra $V_{\hat{\fg}}(\ell,0)$ is an $\CH$-module vertex algebra.
\end{proposition}

\begin{proof} Recall that $V_{\hat{\fg}}(\ell,0)$ is a $\BZ$-graded vertex algebra  and a $\BZ$-graded $\CH$-module.
We claim that $L_{-1}^{(n)}=\D^{(n)}$ and $L_{0}^{(n)}=\binom{-2\deg}{n}$ on $V_{\hat{\fg}}(\ell,0)$
for $n\in \BN$.
For $a\in \fg$, with $a(x)$ viewed as an element of $\hat{\fg}[[x,x^{-1}]]$, we have
\begin{equation*}
e^{zL_{-1}}* a(x)=a(x+z),\   \   \   \
e^{zL_{0}}*a(x)=(1+z)^{-2}a(x(1+z)^{-2}),
\end{equation*}
where  $*$ denotes the $\CH$-module action  on $\hat{\fg}$.
We have
\begin{gather*}
e^{zL_{-1}}a(x)v=(e^{zL_{-1}}* a(x))e^{zL_{-1}}v=a(x+z)e^{zL_{-1}}v, \\
e^{z\D}a(x)v=e^{z\D}Y(a,x)v=Y(a,x+z)e^{z\D}v=a(x+z)e^{z\D}v
\end{gather*}
for $a\in \fg,\  v\in V_{\hat{\fg}}(\ell,0)$. We also have $e^{xL_{-1}}{\bf 1}={\bf 1}=e^{x\D}{\bf 1}$.
It then follows that $e^{zL_{-1}}=e^{z\D}$. Thus $L_{-1}^{(n)}=\D^{(n)}$ for $n\in \BN$. Similarly, we can show
$L_{0}^{(n)}=\binom{-2\deg}{n}$.

The conditions (i) and \eqref{condition-delta1} in Definition \ref{def-good} are clearly satisfied.
For $a\in \fg$, we have
\begin{align*}
    e^{zL_{1}}Y(a,z_0)e^{-zL_{1}}&=\sum_{n\in\BZ}e^{zL_{1}}a(n) e^{-zL_{1}}z_0^{-n-1}\\
    &=\sum_{n\in\BZ} (e^{zL_{1}}*a(n)) z_0^{-n-1}\\
    &=\sum_{n\in\BZ} \sum_{i\in\BN} \binom{-n}{i}a(n+i)z^i z_0^{-n-1}\\
    &=\sum_{m\in\BZ} \sum_{i\in\BN} \binom{-m+i}{i}a(m)z^i z_0^{-m-1+i}\\
    &=\sum_{m\in\BZ} \sum_{i\in\BN} (-1)^{i}\binom{m-1}{i}a(m)z^i z_0^{-m-1+i}\\
    &=(1-zz_0)^{-2}Y(a,z_0/(1-zz_0)).
\end{align*}
On the other hand, we have
\begin{align*}
    Y\bigl(e^{z(1-zz_0)L_{1}}(1-zz_0)^{-2\deg}a,z_0/(1-zz_0)\bigr)
    =(1-zz_0)^{-2}Y(a,z_0/(1-zz_0)),
\end{align*}
noticing that $e^{xL_{1}}a=e^{xL_{1}}(a_{-1}\1)=e^{xL_{1}}(a_{-1})e^{xL_{1}}\1
=\sum_{n\ge 0}x^{n}a_{-1+n}\1 =a$, and $\deg a=1$.
Thus
\begin{equation}
    e^{zL_{1}}Y(a,z_0)e^{-zL_{1}}=Y\bigl(e^{z(1-zz_0)L_{1}}(1-zz_0)^{-2\deg}a,z_0/(1-zz_0)\bigr).
\end{equation}
Since $\fg$ generates $V_{\hat{\fg}}(\ell,0)$ as a vertex algebra,
it follows from Lemma~\ref{th:induction} 
that $V_{\hat{\fg}}(\ell,0)$ is an $\CH$-module vertex algebra.
\end{proof}

Furthermore, we have:

\begin{proposition}\label{affine-Verma-va}
We have $L_{1}^{(n)}V_{\hat\fg}(\ell,0)_{n}=0$ for all $n\ge 1$.
Furthermore, the space of  invariant bilinear forms on $V_{\hat\fg}(\ell,0)$ is one-dimensional and
there exists a symmetric invariant bilinear form on $V_{\hat\fg}(\ell,0)$ with $(\1,\1)=1$.
\end{proposition}

\begin{proof} In view of Corollary \ref{existence-form}, it suffices to prove that
$L_{1}^{(n)}V_{\hat\fg}(\ell,0)_{n}=0$ for all $n\ge 1$.
Recall that for $a\in \fg,\ n,s\in \BN$, we have $L_{1}^{(n)}(a_{-s})=\binom{s}{n}a_{n-s}$. Thus
\begin{align}\label{eEnas}
 L_{1}^{(n)}(a_{-n})=a_0\   \mbox{ and } \   L_{1}^{(n)}(a_{-s})=0\quad\text{if }n>s.
\end{align}
Let $n\in\BZ_+$. Assume $a^{(1)},\ldots,a^{(r)}\in\fg$, $s_1,\ldots,s_r\in\BZ_+$ with $s_1+\cdots+s_r=n$.
Using (\ref{eEnas}) and the fact that $L_{1}^{(n)}\1=0$ and $a_{0}\1=0$ for all $a\in \fg$, we get
\begin{align*}
    L_{1}^{(n)}a^{(1)}_{-s_1}\cdots a^{(r)}_{-s_r}\1=\sum_{n_1+\cdots+n_r=n}(L_{1}^{(n_1)}a^{(1)}_{-s_1})
    \cdots(L_{1}^{(n_r)}a^{(r)}_{-s_r})\1
    =a^{(1)}_0\cdots a^{(r)}_0 \1=0.
\end{align*}
As
\begin{align*}
    V_{\hat\fg}(\ell,0)_n=\spanf\{a^{(1)}_{-s_1}\cdots a^{(r)}_{-s_r}\1\mid r\ge 1,\ a^{(i)}\in\fg,\  s_i\in\BZ_+,\   s_1+\cdots+s_r=n\},
\end{align*}
we conclude $L_{1}^{(n)}V_{\hat\fg}(\ell,0)_{n}=0$. This completes the proof.
\end{proof}

Note that $V_{\hat\fg}(\ell,0)$ is $\BN$-graded with $V_{\hat\fg}(\ell,0)_0=\BF \1$.
From Lemma \ref{maximal-ideal}, there exists a (unique) maximal graded $V_{\hat\fg}(\ell,0)$-submodule $J$
of $V_{\hat\fg}(\ell,0)$ with $\1\notin J$, which is also the maximal (proper) graded ideal of $V_{\hat\fg}(\ell,0)$.
Since $V_{\hat\fg}(\ell,0)$ as a vertex algebra is generated by $\fg$,
it follows that  $J$ actually is the maximal graded (proper) $\hat{\fg}$-submodule of $V_{\hat\fg}(\ell,0)$.
With Proposition \ref{affine-Verma-va}, by Lemma \ref{maximal-ideal} we have that
$J$ is an $\CH$-submodule of $V_{\hat\fg}(\ell,0)$.
Set
\begin{equation*}
L_{\hat{\fg}}(\ell,0)=V_{\hat\fg}(\ell,0)/J,
\end{equation*}
which  is a simple graded $\hat{\fg}$-module and a simple $\CH$-module vertex algebra.
Combining Proposition \ref{affine-Verma-va} with Corollary \ref{existence-form}, we immediately have:

\begin{corollary}
There exists a non-degenerate symmetric invariant bilinear form $(\cdot,\cdot)$
on $L_{\hat{\fg}}(\ell,0)$, which is uniquely determined by the condition $(\1,\1)=1$.
\end{corollary}

Next, we study the Virasoro vertex algebras over $\BF$.
As before, let $\BF$ be a field of a characteristic $p$ with $p\ne 2$.
The Virasoro algebra $\Vir$ over $\BF$ by definition is
the Lie algebra with basis
$\{L_n\mid n\in\BZ\}\cup\{\bc\}$, where
\begin{equation}
    [L_m,L_n]=(m-n)L_{m+n}+\frac{1}{2}\binom{m+1}{3}\delta_{m+n,0}\bc
\end{equation}
for $m,n\in \BZ$ and $[\bc, \Vir]=0$.

Let $\Vir_{\BC}$ denote the Virasoro Lie algebra over $\BC$.
Identify $\Sl_2$ with the Lie subalgebra of $\Vir_{\BC}$ with basis $\{ L_{\pm 1}, L_0\}$.
Then $\Vir_{\BC}$ is naturally a $U(\Sl_2)$-module Lie algebra (through the adjoint action of $\Sl_2$ on $\Vir_{\BC}$):
\begin{align}
L_{-1}^{(r)}(\bc)&=\delta_{r,0}\bc,&L_{1}^{(r)}(\bc)&=\delta_{r,0}\bc,&L_{0}^{(r)}(\bc)&=\delta_{r,0}\bc,\\
L_{-1}^{(r)}(L_n)&=(-1)^{r}\binom{n+1}{r}L_{n-r},&L_{1}^{(r)}(L_n)&=\binom{-n+1}{r}L_{n+r}, &L_{0}^{(r)}(L_n)&=\binom{2n}{r}L_{n}
\end{align}
for $r\in\BN,\ n\in\BZ$.
Denote by $\Vir_{\BZ}$ the $\BZ$-span of $\frac{1}{2}{\bf c}$ and $L_{n}$ for $n\in \BZ$ in $\Vir_{\BC}$,
which is a Lie subring.
Note that $U(\Sl_2)_{\BZ}$ preserves $\Vir_{\BZ}$.
Then $\Vir_{\BZ}$ is a $U(\Sl_2)_{\BZ}$-module Lie ring (recall that $U(\Sl_2)_{\BZ}$
is a subring of $U(\Sl_2)$).
We have
\begin{align}
\Vir=\BF\otimes_{\BZ}\Vir_{\BZ}
\end{align}
and recall that $\CH=\BF\otimes_{\BZ}U(\Sl_2)_{\BZ}$.
It follows that $\Vir$ (over $\BF$) is an $\CH$-module Lie algebra.
Then $U(\Vir)$ is an $\CH$-module algebra by \cite[Lemma~3.2]{LM}.

Let $c\in \BF$. Recall the vertex algebra $V_{\rm Vir}(c,0)$, whose underlying vector space is the $\Vir$-module
$U(\Vir)/J_{c}$, where $J_{c}$ is the left ideal of $U(\Vir)$ generated by
$\bc-c$ and $L_{n}$ for $n\ge -1$.
Notice that $\sum_{n\ge-1}\BF L_n+\BF (\bc-c)$ is an $\CH$-submodule of $U(\Vir)$.
It follows that $J_{c}$ is an $\CH$-submodule.
Then $V_{\rm Vir}(c,0)$ is an $\CH$-module.
Note that $L_{1}^{(r)}\1=\delta_{r,0}\1$ for $r\in\BN$.  Just as with affine Lie algebra $\hat{\fg}$,
$\Vir$ is a $\BZ$-graded $\CH$-module so that $V_{\rm Vir}(c,0)$ is a $\BZ$-graded $\CH$-module.
Furthermore, we have:

\begin{proposition}
Vertex algebra $V_{\rm Vir}(c,0)$ equipped with the
$\CH$-module structure becomes an $\CH$-module vertex  algebra.
\end{proposition}

\begin{proof}
Recall that $V_{\rm Vir}(c,0)$ is an $\BN$-graded vertex algebra. Just as in Proposition \ref{affine-module-va}
with affine vertex algebras, we have  $L_{-1}^{(n)}=\D^{(n)}$ and $L_{0}^{(n)}=\binom{-2\deg}{n}$ on $V_{\Vir}(c,0)$
for $n\in \BN$. The $\CH$-module structure on $V_{\rm Vir}(c,0)$
satisfies \eqref{condition-delta1} of Definition \ref{def-good}.
We then prove \eqref{condition-conj-E}. We have
\begin{align*}
    e^{zL_{1}}Y(\omega,z_0)e^{-zL_{1}}=Y\bigl(e^{z(1-zz_0)L_{1}}(1-zz_0)^{-2\deg}\omega,z_0/(1-zz_0)\bigr),
\end{align*}
because
\begin{align*}
    e^{zL_{1}}Y(\omega,z_0)e^{-zL_{1}}&=\sum_{n\in\BZ}e^{zL_{1}}L_n e^{-zL_{1}}z_0^{-n-2}\\
    &=\sum_{n\in\BZ}e^{zL_{1}}(L_n)z_0^{-n-2}\\
    &=\sum_{n\in\BZ}\sum_{i\in\BN} \binom{-n+1}{i}L_{n+i} z_0^{-n-2} z^i \\
    &=\sum_{m\in\BZ}\sum_{i\in\BN} \binom{-m+i+1}{i}L_m z_0^{-m-2+i} z^i \\
    &=\sum_{m\in\BZ}\sum_{i\in\BN} (-1)^i \binom{m-2}{i}L_m z_0^{-m-2+i} z^i\\
    &=(1-zz_0)^{-4}L(z_0/(1-zz_0))
\end{align*}
and
\begin{align*}
    Y\bigl(e^{z(1-zz_0)L_{1}}(1-zz_0)^{-2\deg}\omega,z_0/(1-zz_0)\bigr)
    &=(1-zz_0)^{-4}Y\bigl(\omega,z_0/(1-zz_0)\bigr)\\
    &=(1-zz_0)^{-4}L(z_0/(1-zz_0)).
\end{align*}
Since $\omega$ generates $V_{\rm Vir}(c,0)$ as a vertex algebra,
by Lemma~\ref{th:induction}, $V_{\rm Vir}(c,0)$ is  an $\CH$-module vertex  algebra.
\end{proof}

Just as with $V_{\hat\fg}(\ell,0)$, we have:

\begin{lemma}\label{virasoro-vanish}
For $n\ge 1$, we have $L_{1}^{(n)}V_{\rm Vir}(c,0)_{n}=0$ .
\end{lemma}

\begin{proof} We prove by induction on $r$ that
for any ordered $r$-tuples  $(n_1,\ldots,n_r), (s_1,\ldots,s_r)\in \BZ_+^{r}$  with
$n_1+\cdots+n_r=s_1+\cdots+s_r$, we have
\begin{align}\label{temp0.2}
    (L_{1}^{(n_1)}L_{-s_1})\cdots(L_{1}^{(n_r)}L_{-s_r})\1= 0.
\end{align}
Recall that for $n,s\in\BZ_+$, we have
\begin{align}\label{eq:Vir-BF-1}
    L_{1}^{(n)}(L_{-s})=\binom{s+1}{n}L_{n-s}.
\end{align}
In particular, this implies that  $L_{1}^{(n)}(L_{-s})=0$ if $n> s+1$. For any positive integer $n$, we have
$(L_{1}^{(n)}L_{-n})\1=(n+1)L_{0}\1= 0$. This establishes the base case.

 Now, assume $r\ge 2$. Let $(n_1,\ldots,n_r),\ (s_1,\ldots,s_r)\in\BZ_{+}^{r}$ such that
$n_1+\cdots+n_r=s_1+\cdots+s_r$. Recall that $L_{1}^{(1)}L_{1}^{(k)}=(k+1)L_{1}^{(k+1)}$ for $k\in \BN$.
 If $n_i>s_i+1$ for some $1\le i\le r$, we have $(L_{1}^{(n_i)}L_{-s_i})=0$, so that (\ref{temp0.2}) holds.
 It remains to consider the case where $n_i\le s_i+1$ for all $1\le i\le r$.
 Suppose $n_i=s_i+1$ for some $1\le i\le r$. Then $(L_{1}^{(n_i)}L_{-s_i})=L_1$.
 Noting that  $L_1$ acts as $L_{1}^{(1)}$ on $V_{\rm Vir}(c,0)$, we have
\begin{align*}
    &\phantom{=}\;\;(L_{1}^{(n_1)}L_{-s_1})\cdots(L_{1}^{(n_r)}L_{-s_r})\1\\
    &=(L_{1}^{(n_1)}L_{-s_1})\cdots(L_{1}^{(n_{i-1})}L_{-s_{i-1}})L_1(L_{1}^{(n_{i+1})}L_{-s_{i+1}})
    \cdots(L_{1}^{(n_r)}L_{-s_r})\1,
\end{align*}
where
\begin{align*}
&\phantom{=}\;\;L_{1}(L_{1}^{(n_{i+1})}L_{-s_{i+1}})\cdots(L_{1}^{(n_r)}L_{-s_r})\1 \\
&=\sum_{j=1}^{r-i-1}(L_{1}^{(n_{i+1})}L_{-s_{i+1}})\cdots(L_{1}^{(1)}L_{1}^{(n_{i+j})}L_{-s_{i+j}})\cdots(L_{1}^{(n_r)}L_{-s_r})\1 \\
&=\sum_{j=1}^{r-i-1}(n_{i+j}+1)(L_{1}^{(n_{i+1})}L_{-s_{i+1}})\cdots(L_{1}^{(n_{i+j}+1)}L_{-s_{i+j}})\cdots(L_{1}^{(n_r)}L_{-s_r})\1.
\end{align*}
From induction hypothesis,  we have
\begin{equation*}
(L_{1}^{(n_1)}L_{-s_1})\cdots(L_{1}^{(n_{i-1})}L_{-s_{i-1}})(L_{1}^{(n_{i+1})}L_{-s_{i+1}})\cdots(L_{1}^{(n_{i+j}+1)}L_{-s_{i+j}})\cdots(L_{1}^{(n_r)}L_{-s_r})\1=0
\end{equation*}
for all $1\le j\le r-i-1$.
Finally, we are left with the case where $n_i\le s_i$ for $i=1,\ldots,r$.
 Since $n_1+\cdots+n_r=s_1+\cdots+s_r$, we must have $n_i=s_i$ for $i=1,\ldots,r$.
In this case, we also have
\begin{align*}
 (L_{1}^{(n_1)}L_{-s_1})\cdots(L_{1}^{(n_r)}L_{-s_r})\1
 =   (L_{1}^{(n_1)}L_{-n_1})\cdots(L_{1}^{(n_r)}L_{-n_r})\1=(n_1+1)\cdots (n_r+1) L_{0}^{r}\1=0.
\end{align*}
This completes the induction.
\end{proof}

Now, combining Lemma \ref{virasoro-vanish} with Corollary \ref{existence-form}, we immediately have:

\begin{corollary}
The space of invariant bilinear forms on $V_{\rm Vir}(c,0)$ is one-dimensional  and
there exists a symmetric invariant bilinear form $(\cdot,\cdot)$,
which is uniquely determined by the condition $(\1,\1)=1$.
\end{corollary}

Let $J$ denote the maximal submodule of the adjoint $V_{\rm Vir}(c,0)$-module.
It follows from Lemmas \ref{maximal-ideal} and \ref{virasoro-vanish} that $J$ is an ideal and an $\CH$-submodule.
Consequently,  the quotient $\CH$-module vertex algebra of $V_{\rm Vir}(c,0)$ modulo $J$, which is
denoted by $L_{\rm Vir}(c,0)$, is a simple $\CH$-module vertex algebra.
Again, by Lemma \ref{virasoro-vanish} and Corollary \ref{existence-form} we immediately have:

\begin{corollary}
On the simple $\CH$-module vertex algebra $L_{\rm Vir}(c,0)$,
there exists a non-degenerate symmetric invariant bilinear form
$(\cdot,\cdot)$, which is uniquely determined by the condition $(\1,\1)=1$.
\end{corollary}


\begin{thebibliography}{FLM}

\bibitem[B1]{B86}
R.E. Borcherds, Vertex algebras, Kac-Moody algebras, and the Monster,
    \textit{Proc. Nat. Acad. Sci. U.S.A.} \textbf{83}(10) (1986) 3068--3071.

\bibitem[B2]{BR2}
R.E. Borcherds, Modular moonshine III. \textit{Duke Math. J.} \textbf{93}(1) (1998) 129--154.

\bibitem[BR]{BR1}
R.E. Borcherds, A. Ryba, Modular moonshine II. \textit{Duke Math. J.} \textbf{83}(2) (1996) 435--459.

\bibitem[DG]{DG}
C. Dong, R.L. Griess Jr, Integral forms in vertex operator algebras which are invariant under finite groups,
    \textit{J. Algebra} \textbf{365} (2012) 184--198.

\bibitem[DLR]{DLR}
C. Dong, C.H. Lam, L. Ren, Modular framed vertex operator algebras, preprint, arXiv:1709.04167, 2017.

\bibitem[DR1]{DR1} C. Dong, L. Ren, Representations of vertex operator algebras over an arbitrary field,
    \textit{J. Algebra} \textbf{403} (2014) 497--516.

\bibitem[DR2]{DR2} C. Dong, L. Ren, Vertex operator algebras associated to the Virasoro algebra over an arbitrary field,
    \textit{Trans. Amer. Math. Soc.} \textbf{368}(7) (2016) 5177--5196.

\bibitem[FHL]{FHL}
    I.B. Frenkel, Y.-Z. Huang,  J. Lepowsky,
    On axiomatic approaches to vertex operator algebras and modules,
    \textit{Mem. Amer. Math. Soc.} \textbf{104}(494), 1993.

\bibitem[FLM]{FLM} I. Frenkel, J. Lepowsky, A. Meurman,
    \textit{Vertex Operator Algebras and the Monster}.
    Pure Appl. Math., 134. Academic Press, Inc., Boston, MA, 1988.

\bibitem[GL]{GL} R.L. Griess Jr, C.H. Lam, Groups of Lie type, vertex algebras, and modular moonshine,
    \textit{Int. Math. Res. Not. IMRN} \textbf{2015}(21) (2015) 10716--10755.

\bibitem[H]{hum}
J.E. Humphreys, \textit{Introduction to Lie Algebras and Representation Theory}, Grad. Texts in Math., \textbf{9}.
Springer-Verlag,  New York, 1972.

\bibitem[LL]{LL}
J. Lepowsky, H.-S. Li,
    \textit{Introduction to Vertex Operator Algebras and Their Representations}.
    Progr. Math., 227. Birkh\"{a}user Boston, Inc., Boston, MA, 2004.

\bibitem[L]{Li94} H.-S. Li, Symmetric invariant bilinear forms on vertex operator algebras,
    \textit{J. Pure Appl. Algebra} \textbf{96}(3) (1994), 279--297.

\bibitem[LM]{LM}
H.-S. Li, Q. Mu, Heisenberg VOAs over fields of prime characteristic and their representations,
    preprint, arXiv:1501.04314; to appear in \textit{Trans. Amer. Math. Soc.}

\bibitem[M1]{Mc1}
R. McRae, On integral forms for vertex algebras associated with affine Lie algebras and lattices,
\textit{J. Pure Appl. Algebra} \textbf{219}(4) (2015) 1236--1257.

\bibitem[M2]{Mc2}
R. McRae, Intertwining operators among modules for affine Lie algebra
and lattice vertex operator algebras which respect integral forms,
\textit{J. Pure Appl. Algebra} \textbf{219}(10) (2015) 4757--4781.

\bibitem[R]{R}
L. Ren, Modular $A_n(V)$ theory,
    \textit{J. Algebra} \textbf{485} (2017) 254--268.
\end{thebibliography}
\end{document}